\documentclass[11pt]{amsart}
\usepackage{mathptmx, bbm, amscd, amssymb, enumerate, colonequals, mathdots, xcolor, tikz-cd,mathtools, amsthm, caption}
\usepackage{color}
\definecolor{chianti}{rgb}{0.6,0,0}
\definecolor{meretale}{rgb}{0,0,.6}
\definecolor{leaf}{rgb}{0,.35,0}
\usepackage[colorlinks=true, pagebackref, hyperindex, citecolor=meretale, urlcolor=leaf, linkcolor=chianti]{hyperref}

\usepackage{tikz}
\usetikzlibrary{decorations.markings,intersections,positioning,calc}

  \tikzset{mylabel/.style  args={at #1 #2  with #3}{
    postaction={decorate,
    decoration={
      markings,
      mark= at position #1
      with  \node [#2] {#3};
 } } } }

\DeclareMathOperator{\depth}{depth}
\DeclareMathOperator{\height}{ht}
\DeclareMathOperator{\Tr}{Tr}

\DeclareMathOperator{\ini}{in}

\DeclareMathOperator{\Ass}{Ass}

\newcommand{\fraka}{{\mathfrak{a}}}

\newcommand{\frakp}{{\mathfrak{p}}}

\newcommand{\calR}{\mathcal{R}}
\newcommand{\gr}{\mathfrak{gr}}

\def\NN{\mathbb{N}}

\usepackage{thmtools}
\usepackage{thm-restate}

\newtheorem*{thm*}{Theorem}

\newtheorem{thm}{Theorem}[section]
\newtheorem{lem}[thm]{Lemma}
\newtheorem{cor}[thm]{Corollary}
\newtheorem{prop}[thm]{Proposition}

\theoremstyle{definition}
\newtheorem{defn}[thm]{Definition}

\declaretheorem[
  style=definition,
  title=Example,
  refname={example,examples},
  Refname={Example,Examples},
  sharenumber=thm,
]{exa}

\declaretheorem[
  style=definition,
  title=Remark,
  sharenumber=thm,
]{rmk}

\declaretheorem[
  style=definition,
  title=Setup,
  sharenumber=thm,
]{setup}

\begin{document}
\title[Symbolic powers of the generic linkage of maximal minors]{Symbolic powers of the generic linkage of \\ maximal minors}

\author{Vaibhav Pandey}
\address{Department of Mathematics, Purdue University, 150 N University St., West Lafayette, IN~47907, USA}
\email{pandey94@purdue.edu}
\author{Matteo Varbaro}
\address{Dipartimento di Matematica, Universita di Genova Via Dodecaneso, 35 16146, Genova, Italy}
\email{matteo.varbaro@unige.it}

\thanks{V.P. is partially supported by the AMS--Simons Travel Grant. M.V. is supported by PRIN~2020355B8Y ``Squarefree Gr\"obner degenerations, special varieties and related topics,'' by MIUR Excellence Department Project awarded to the Dept.~of Mathematics, Univ.~of Genova, CUP D33C23001110001, by INdAM-GNSAGA and by NSF grant DMS-1928930 and by Alfred P. Sloan Foundation grant G-2021-16778.}


\begin{abstract}
Let $I$ be the ideal generated by the maximal minors of a matrix of indeterminates over a field and let $J$ denote the generic link, i.e., the most general link, of $I$. The generators of the ideal $J$ are not known in general. We provide an explicit description of the lead terms of the generators of $J$ using Gr\"obner degeneration. Indeed, we construct a degeneration which preserves the entire graded Betti table of $J$ on passing to the initial ideal. We leverage this construction to establish the equality of the symbolic and ordinary powers of $J$. Our analysis of the initial ideal readily yields the Gorenstein property of the associated graded ring of $J$, and, in positive characteristic, the $F$-rationality of the Rees algebra of $J$. Using the technique of $F$-split filtrations, we further establish the $F$-regularity of these blowup algebras.     
\end{abstract}

\maketitle

%%%%%%%%%%%%%%%%%%%%%%%%%%%%%%%%%%%%%%
\section{Introduction}
%%%%%%%%%%%%%%%%%%%%%%%%%%%%%%%%%%%%%%%

Two proper homogeneous ideals $I$ and $J$ in a polynomial ring $R$ over a field are said to be \textit{linked} if there exists an ideal $\fraka$ generated by a regular sequence such that 
 \[J = \mathfrak{a}:I \qquad \text{and} \qquad I = \mathfrak{a}:J.\]
It is easily seen that all the associated primes of $I$ and $J$ have the same height and if none of these associated primes are shared between $I$ and $J$, the definition says exactly that the scheme-theoretic union $V(I) \cup V(J)$ of the vanishing loci is $V(\fraka)$, so $V(J)$ `links' $V(I)$ to a complete intersection. Classically, linkage was used as a tool to classify curves in $\mathbb{P}^3$ by linking a given curve to a simpler one. For instance, the twisted cubic is linked to a line in $\mathbb{P}^3$. 

By the seminal work of Peskine and Szpiro \cite{PS74}, and the extensive subsequent contributions of Huneke and Ulrich \cite{HunekeSCM, Huneke-Ulrich85, Huneke-Ulrich87, HunekeUlrich-Duke, Huneke-Ulrich88} in developing the subject from an algebraic perspective, the modern study of linkage has evolved as a duality theory between unmixed ideals of the same height. For example, the canonical module of $R/I$ is $J/\fraka$ and, by symmetry, that of $R/J$ is $I/\fraka$. Also, $R/I$ is a Cohen--Macaulay ring if and only if $R/J$ is.

Evidently, the link depends on the choice of the regular sequence. The \textit{generic link} of $I$ is the link obtained on choosing the regular sequence to be a generic combination of the generators of $I$ (see Definition \ref{defn:genericlink}). At least when $R/I$ is a Cohen--Macaulay ring, the generic link acts as the prototypical link in the sense that it can be specialized to any other link of $I$. Generic links are of fundamental importance for the structural results of ideals linked to complete intersections (\textit{licci} ideals) \cite[Corollary 4.9]{Huneke-Ulrich87}, divisor class groups of rigid Cohen--Macaulay normal domains \cite[Theorem 4.2]{Huneke-Ulrich85}, Castelnuovo--Mumford regularity of projective
schemes having nice singularities \cite[Theorem 4,4]{Chardin-Ulrich}, etc.  

Let $X$ be a matrix of indeterminates over a field $K$ and $I$ be the ideal of $R=K[X]$ generated by the maximal minors of $X$. Let $J$ be the generic link of $I$ in a polynomial extension $S$ of $R$. The generic determinantal ring $R/I$ is an \textit{Algebra with a Straightening Law} (ASL) \cite{DEP}. This allows for the use of the \textit{standard monomial theory} to get the equality of the symbolic and ordinary powers of the ideal $I$ of maximal minors. It is not at all clear whether the generic link $S/J$ has an ASL structure. In addition, $J$ is typically not generated in a single degree. One of the main results of this paper is

\begin{thm*} [Corollary \ref{cor:OrdinaryEqualsSymbolic}]
The symbolic and ordinary powers of the generic link $J$ are equal.
\end{thm*}

The main difficulty in proving the above theorem is that the generators of the link $J$ are not known. In principle, a mapping cone constructed by Ferrand resolves the link starting from the Eagon--Northcott resolution of $I$; however it is very complicated to obtain the explicit generators of $J$. The limitation of the homological viewpoint stems from the fact that the comparison maps in the mapping cone are difficult to determine explicitly (see \cite{Hema} for a DG-algebra approach). Since the generators of links are rarely known, the overarching \textbf{theme of this paper} is to demonstrate that linkage theory provides a fertile ground for the meaningful application of Gr\"obner degeneration techniques.

In Section \ref{Section2}, we discuss some techniques to determine suitable initial ideals of the generic link $J$. We show that for a carefully chosen term order in $S$ (see Example \ref{example:term-order}), the reduced Gr\"obner basis of $J$ is a minimal generating set of $J$ and the initial ideal of $J$ is squarefree (Corollary \ref{cor:GrobnerBasis}). In fact, using \cite{CoVa} and the techniques to study the generic link discussed in \cite[\S 3]{Pandey}, we show that there is `no loss of information' in passing to the initial ideal (Lemma \ref{lem:main}):
\[\beta_{i,j}(S/J)=\beta_{i,j}(S/\ini(J)) \quad \text{for all}\quad i,j\geq 0.\]
This toolkit works in greater generality and may apply equally well to study the generic links of certain other classes of ideals (see the comment after Remark \ref{remark:ForP}).

In Section \ref{Section3}, we give a precise description of the initial ideal $\ini(J)$ of the generic link (Theorem \ref{thm:in(J)}). It turns out that the generators of $\ini(J)$ arise from intriguing combinatorial patterns (Remark \ref{rmk:lead terms}). With this explicit description of $\ini(J)$ at our disposal, we prove Corollary \ref{cor:OrdinaryEqualsSymbolic} in Section \ref{Section4} by establishing:

\begin{thm*}[Theorem \ref{thm:main}]
For a suitable term order, the symbolic and ordinary powers of the initial ideal $\ini(J)$ of the generic link are equal.    
\end{thm*}

Quite generally, if the initial ideal is squarefree, then Theorem \ref{thm:main} is stronger than Corollary \ref{cor:OrdinaryEqualsSymbolic}; this has some beautiful consequences of general interest (see Proposition \ref{prop:crucialmethod}). Theorem \ref{thm:main} also has connections with Linear Programming: an equivalent formulation is that the incidence matrix of the clutter whose edges correspond to the minimal generators of $\ini(J)$ satisfies the so called \textit{max-flow-min-cut} property by a theorem of Herzog, Hibi, Trung, and Zheng \cite[Theorem 1.4]{HHTZ08}. The proof of Theorem \ref{thm:main} requires the use of a criterion of Monta\~no, and  N\'u\~nez-Betancourt (see \cite[Corollary 4,10]{Jonathon-Luis}) for the equality of the symbolic and ordinary powers of squarefree monomial ideals. We point out that one of the main complications in the proof arises from the fact that the symbolic and ordinary powers of the ideal of the `nontrivial generators' of $\ini(J)$ are usually \textit{not} equal (see Example \ref{exU}, Corollary \ref{cor:NontrivialGenerators}). This forces us to make a number of careful reductions in order to keep a check on the underlying combinatorial arguments.  

The equality of the symbolic and ordinary powers of an ideal is intimately connected with the singularities of its blowup algebras \cite{Hoc, HunekeIllinois, Huneke-Simis-Vasconcelos}. One of the motivations of this work is the following question that Bernd Ulrich asked us at the special semester in Commutative Algebra in SLMath, Berkeley (Spring 2024): ``Is the Rees algebra $\mathcal{R}(J)$ of the generic link Cohen--Macaulay?'' In Section \ref{Section5}, we study the singularities of the blowup algebras of the generic link. Our analysis of the initial ideal $\ini(J)$ readily yields the Gorenstein property of the associated graded ring $\gr(J)$ of the generic link as well as the $F$-rationality of its Rees algebra $\mathcal{R}(J)$ in positive prime characteristic (and rational singularities in characteristic zero) (see Corollary \ref{cor:Rees}). In particular, $\mathcal{R}(J)$ is a Cohen--Macaulay normal domain, thus answering Ulrich's question in the affirmative. 

The theory of \textit{$F$-split filtrations} \cite[\S 4]{dS-Mn-NB} developed recently by De Stefani, Monta\~no, and  N\'u\~nez-Betancourt provides an efficient technique to study the singularities of blowup algebras in prime characteristic, especially for ideals which are determinantal in nature. Indeed, using this technique, we further obtain the $F$-regularity of $\mathcal{R}(J)$ and $\gr(J)$ in Theorem \ref{thn:symbolicFsplit}, thereby strengthening Corollary \ref{cor:Rees}. Summarizing this discussion:

\begin{thm*}
Over an $F$-finite field of positive characteristic, the Rees algebra $\mathcal{R}(J)$ and the associated graded ring $\gr(J)$ are strongly $F$-regular. In characteristic $0$, they have rational singularities. In any characteristic, $\mathcal{R}(J)$ is a Cohen--Macaulay normal domain and $\gr(J)$ is a Gorenstein normal domain.
\end{thm*}

With the blowup algebras of the generic link of maximal minors well-understood, it is natural to ask what happens in the case of non-maximal minors. While we do not know the answer in general, and calculations with computer algebra systems seem to be quite inaccessible, we remark that for the generic link of minors of \textit{any} size of a generic square matrix, the equality of the symbolic and ordinary powers, the Cohen--Macaulay property of the Rees algebra, and the Gorenstein property of the associated graded ring, all continue to hold. This is essentially due to the fact that the generic link in this case is an \textit{almost complete intersection} (Corollary \ref{cor:non-maximal}). 

%%%%%%%%%%%%%%%%%%%%%%%%%%%%%
\section{Gr\"obner degeneration of links}\label{Section2}
%%%%%%%%%%%%%%%%%%%%%%%%%%%%%%

Throughout the paper, $R$ will denote a polynomial ring over a field $K$, equipped with the standard grading.
\begin{defn}\label{defn:Linkage}
    Let $I$ and $J$ be proper ideals of $R$. We say that $I$ and $J$ are \emph{linked} (or $R/I$ and $R/J$ are linked) if there exists an ideal $\mathfrak{a}\subset R$ generated by a regular sequence such that \[J = \mathfrak{a}:I \qquad \text{and} \qquad I = \mathfrak{a}:J,\] and use the notation $I \sim _{\mathfrak{a}} J$. Furthermore, we say that the link is \emph{geometric} if we have $\height(I+J) \geq \height(I)+1$. 
\end{defn}

It is clear that the ideal $\mathfrak{a}$ is contained in $I$ and $J$. Note that the associated primes of $I$ and $J$ have the same height, that is, the ideals $I$ and $J$ are \textit{unmixed}. Also the heights of the ideals $I$, $J$, and $\mathfrak{a}$ are equal. Moreover, when the link is geometric, it follows that the ideal $\mathfrak{a}$ is the intersection of $I$ and $J$. We now define the generic link, our main object of study.

\begin{defn}\label{defn:genericlink}
    Let $I\subset R$ be an unmixed ideal of height $g>0$. Let $f_1,\ldots,f_r$ be a generating set of $I$. Let $Y$ be a $g \times
   r$ matrix of indeterminates, and let $\mathfrak{a}\subset R[Y]$ be the ideal generated by the entries of the matrix $Y[f_1\dots f_
   r]^T$ (where $[\quad]^T$ denotes the transpose of the matrix). The ideal \[J = \mathfrak{a}:IR[Y]\] is the \emph{generic link} of $I$.
\end{defn}

\begin{rmk}
The generic link is essentially independent of the choice of a generating set of the ideal: Any two generic links of an ideal are isomorphic up to adjoining finitely many indeterminates \cite[Proposition 2.4]{Huneke-Ulrich85}. However, to be precise, we will always specify the generating set of $I$ that we are working with.
\end{rmk}

The following example motivates our choice of term order in studying the initial ideal of the generic link:

\begin{exa}\label{example:term-order}
Let $X\colonequals (x_{i,j})$ be a $2\times 4$ matrix of indeterminates over a field $K$, and $I$ be the ideal generated by the $2$-minors of $X$ in $R= K[X]$. Let $Y\colonequals (Y_{i,j})$ be a $3 \times 6$ matrix of indeterminates, $S= R[Y]$, and $J = \mathfrak{a}:IS$ be the generic link constructed considering the 2-minors of $X$ as the generating set of $I$. Consider the reverse lexicographical term order on $S$ given by
\[Y_{3,6}>Y_{3,5}>\cdots Y_{3,1}>Y_{2,6}>\cdots >Y_{1,1}>x_{2,4}>x_{2,3}> \cdots>x_{2,1}>x_{1,4}>\cdots>x_{1,1}.\] 
By Macaulay$2$ \cite{Macaulay2}, the graded Betti numbers of $J$ and $\ini(J)$ are as follows:

\[\centering
\begin{tabular}{llllll} \hline
$\beta(S/J)$ & 0 & 1 & 2 & 3 \\ \hline
0 & 1 & - & - & - \\ \hline
1 & - & - & - & - \\ \hline
2 & - & 3 & - & - \\ \hline
3 & - & - & - & - \\ \hline
4 & - & 3 & 11 & 6 \\ \hline
\end{tabular} \qquad \begin{tabular}{lllllll} \hline
$\beta(S/\ini(J))$ & 0 & 1 & 2 & 3 & 4 \\ \hline
0 & 1 & - & - & - & - \\ \hline
1 & - & - & - & - & - \\ \hline
2 & - & 3 & 3 & 1 & -\\ \hline
3 & - & 3 & 5 & 2 & -\\ \hline
4 & - & 7 & 13 & 7 & 1 \\ \hline
5 & - & - & 1 & 1 & - \\ \hline
\end{tabular}\]

Note that passing to the initial ideal with respect to this term order does \textit{not} preserve homological data well. The initial ideal $\ini(J)$ picks up several new syzygies; in particular, this term order produces an inefficient Gr\"obner basis: $\ini(J)$ picks up $3$ new minimal generators in degree $4$, and $4$ new minimal generators in degree $5$. In addition, even though $\ini(I)$ is squarefree (the minors form a Gr\"obner basis), $\ini(J)$ is \textit{not} squarefree. Indeed, the monomials
\[x_{1,1}^2Y_{1,4}Y_{2,2}Y_{3,1} \quad \text{and} \quad x_{2,1}^2Y_{1,4}Y_{2,2}Y_{3,1}\]
are minimal generators of $\ini(J)$.

Instead, consider the term order $<_1$ on $S$ described as after Lemma \ref{lemma:initialIsLink}. Working with this term order, we get that the corresponding initial ideal $\ini_{<_1}(J)$ is squarefree (see Proposition \ref{p:radical}) and the reduced Gr\"obner basis of $J$ is a minimal set of generators of $J$ (see Corollary \ref{cor:GrobnerBasis}). In fact, we show in Lemma \ref{lem:main} the much stronger property that the graded Betti table of the link is preserved on passing to the initial ideal:
\[\beta_{i,j}(S/J)=\beta_{i,j}(S/\ini_{<_1}(J)) \quad \text{for each}\quad i,j\geq 0.\]
\end{exa}

The following property was defined by the first author \cite[Definition 3.1]{Pandey}:

\begin{defn}\label{defn:PropertyP}
Let $I\subset R$ be an unmixed ideal. Let $\mathbf{P}$ be the following property of the pair $(R,I)$: For a fixed term order $<$ in $R$, there exist elements $\underline{\alpha} \colonequals \alpha_1, \ldots, \alpha_{\height(I)}$ in the ideal $I$ such that each initial term $\ini_<(\alpha_i)$ is squarefree and each pair of initial terms $\ini_<(\alpha_i), \ini_<(\alpha_j)$ is mutually coprime for $i \neq j$.
\end{defn}

\begin{rmk} \label{remark:ForP}
Since each pair of initial terms $\ini_<(\alpha_i), \ini_<(\alpha_j)$ is mutually coprime, the monomials $\ini_<(\alpha_1), \ldots, \ini_<(\alpha_{\height(I)})$ form an $R$-regular sequence. Therefore the polynomials $\underline{\alpha}= \alpha_1, \ldots, \alpha_{\height(I)}$ also form an $R$-regular sequence (see, for example, \cite[Proposition 1.2.12]{BCRV}).     

We caution the reader that, in general, the regular sequence $\underline{\alpha}$ may\textit{ not} be part of a minimal generating set of $I$. For example, take $I=(x_1x_2+x_4x_5, \ x_1x_3)$ in $K[x_1,\ldots ,x_5]$ and the lexicographical term order $x_1>\cdots >x_5$.
\end{rmk}

The aim of this section is to provide a technique (see Lemma \ref{lem:main}, Corollary \ref{cor:GrobnerBasis}) to compute the initial ideal of the generic link for ideals $I\subset R$ such that $(R,I)$ has the property $\mathbf{P}$ and $I$ has a linear resolution. Although in this paper we focus on the ideal of maximal minors---which has property $\mathbf{P}$ (e.g. see \cite[Proposition 4.6]{Pandey})---we remark that our techniques may also apply to other important cases in linkage theory (see \cite[\S3, \S5]{Pandey}) like the generic links of generic height $3$ Gorenstein ideals, the generic residual intersections (a generalization of linkage) of complete intersection ideals, some monomial ideals, etc.  

The following is a particular case of \cite[Theorem 3.13]{Varbaro-Koley} if $I$ is a radical ideal; we think it is useful to underline that in our setting, the assumption that $I$ is radical is not necessary.

\begin{prop}\label{p:radical}
Let $I\subset R$ be an ideal. Suppose that the pair $(R,I)$ has the property $\mathbf{P}$ with respect to a term order $<$. Then $\ini_<(I)$ is a squarefree monomial ideal. In particular, $I$ is radical.
\end{prop}
\begin{proof}
By \cite[Theorem 3.13]{Varbaro-Koley}, it suffices to show that $I$ is radical. Indeed, since $I$ is unmixed of height $g>0$ (say), $g$ is also the maximum among the heights of the minimal primes of $I$, and if $\alpha_1, \ldots, \alpha_g$ is the $R$-regular sequence contained in $I$ given by the property $\mathbf{P}$ of $(R,I)$, then the product $\alpha=\prod_{i=1}^g\alpha_i\in I^g\subset I^{(g)}$ and $\ini_<(\alpha)=\prod_{i=1}^g\ini_<(\alpha_i)$ is squarefree.

To show that $I$ is radical, first assume that $K$ has positive characteristic. Since $R/(\underline{\alpha})$ is clearly $F$-pure, we get that $R/I$ is $F$-pure by \cite[Corollary 3.3]{PT24} (note that the unmixedness of $I$ implies that $I=(\underline{\alpha}):((\underline{\alpha}):I)$). In particular $R/I$ is reduced.

If $K$ has characteristic $0$, a reduction modulo $p$ argument is possible. By \cite[Lemma 2.5]{Seccia}, and using the same notation, we can choose a finitely generated $\mathbb{Z}$-algebra $A$ such that, if $R=K[x_1,\ldots ,x_n]$, $R'=A[x_1,\ldots ,x_n]$ and $I'=I\cap R'$, $\underline{\alpha}'=\underline{\alpha}\cap R'$, we have $I=I'R$, $\underline{\alpha}=\underline{\alpha}'R$, $\ini_<(I)(p)=\ini_<(I(p))$ and $\ini_<(\underline{\alpha})(p)=\ini_<(\underline{\alpha}(p))$ for any prime number $p\gg 0$. We still have $\underline{\alpha}(p)\subset I(p)$ is a regular sequence of length $g$ for all $p\gg 0$, so if we show that $I(p)$ is unmixed of height $g$ for all $p\gg 0$ we can conclude, because by the previous part $\ini_<(I)(p)=\ini_<(I(p))$ would be squarefree for all $p\gg 0$, hence $\ini_<(I)$ would be squarefree.

To prove that $I(p)$ is unmixed of height $g$, we need to show that the Krull dimension of $\mathrm{Ext}_{R_p}^i(R_p/I(p),R_p)$ is less than $n-i$ for any $i\neq g$ (see \cite[Lemma 2.3.10]{BCRV}). Note that $\mathrm{Ext}_{R_p}^i(R_p/I(p),R_p)$ can be computed by a finite free resolution of $R_p/I(p)$, which for $p\gg 0$ is the same as a finite free resolution of $R/I$ modulo $p$ and $\mathrm{Ext}_{R_p}^i(R_p/I(p),R_p)\cong \mathrm{Ext}_{R}^i(R/I,R)(p)$ for all $p\gg 0$. The latter has Krull dimension less than $n-i$ whenever $i\neq g$ because $\dim \mathrm{Ext}_{R}^i(R/I,R)<n-i$ for all $i\neq g$ since $I$ is unmixed. This finishes the proof.
\end{proof}

\begin{lem}\label{lemma:initialIsLink}
Let $I\subset R$ be an unmixed ideal such that the pair $(R,I)$ has the property $\mathbf{P}$. If $I$ and $J$ are linked by the regular sequence $\underline{\alpha}$, then both $I \sim _{(\underline{\alpha})}J$ and $\ini(I) \sim _{(\ini(\underline{\alpha}))} \ini(J)$ are geometric links.
\end{lem}

\begin{proof}
We first show that $I$ and $J$ are linked geometrically by the regular sequence $\underline{\alpha}$. We can write $(\underline{\alpha})=\bigcap_{i=1}^k p_i$ as an intersection of prime ideals since $(\alpha)$ is radical. Then
$$I=(\underline{\alpha}):J = \bigcap_{i=1}^k(p_i:J)=\{\bigcap p_i : J \nsubseteq p_i \}$$
is a radical ideal. By symmetry, $J = \{\bigcap p_i: I \nsubseteq p_i\}$ is also radical. From this description, it is clear that $I$ and $J$ cannot share any associated prime ideals, i.e., they are geometrically linked. We have, $(\underline{\alpha}) = I \cap J$, so
\[ (\ini(\underline{\alpha})) = \ini(I\cap J) = \ini(I) \cap \ini(J),\]
where the last equality holds by \cite[Proposition 2.1.12]{BCRV} since the ideal $(\ini(\underline{\alpha}))$ is squarefree. Furthermore, both $\ini(I)$ and $\ini(J)$ are squarefree by Proposition \ref{p:radical}, so we get that the initial ideals $\ini(I)$ and $\ini(J)$ are also unmixed (see \cite[Corollary 2.1.8]{BCRV}). Finally, since $\ini(\underline{\alpha})$ is a squarefree complete intersection, repeating the argument in the beginning of the proof yields that the ideals $\ini(I)$ and $\ini(J)$ are also linked geometrically .   
\end{proof}

In \cite{Pandey}, the first author proved that $\mathbf{P}$ propagates along generic links. To state this result precisely, we need to introduce a term order $<_1$ on $R[Y]$ starting from a term order $<$ on $R$. Given a monomial $u\in R[Y]$, we write $u_x\in R$ for the image of $u$ under the map of $R$-algebras from $R[Y]$ to $R$ sending each $Y_{i,j}$ to $1$. Define the following linear order on the variables of $R[Y]$:
\[ Y_{1,1}>Y_{2,2}>\cdots >Y_{g,g}> \text{ the remaining } Y_{i,j}> x_k, \]
where the order $<$ on the indeterminates $x_k$ is the one given in the property $\mathbf{P}$ and the order on ``the remaining $Y_{i,j}$" is arbitrary. Consider the following term order $<_1$ on $R[Y]$: If $u$ and $v$ are monomials of $R[Y]$,
\[u<_1 v \iff \begin{cases}
    u/u_x \mbox{ is lexicographically smaller than }v/v_x, \\
    u/u_x=v/v_x \mbox{ and }u_x<v_x.
\end{cases}\]

\begin{lem}\cite[Lemma 3.3]{Pandey} \label{lemma:Pinherited}
Let $I\subset R$ be an unmixed ideal of height $g>0$. If the pair $(R,I)$ has the property $\mathbf{P}$ for some term order $<$ on $R$, then the pair $(R[Y],J)$, where $J$ is the generic link of $I$ computed by choosing a generating set $\alpha_1,\ldots ,\alpha_r$ of $I$ containing $\underline{\alpha}$, also has $\mathbf{P}$ with respect to the term order $<_1$ on $R[Y]$. 

More precisely, if $Y$ is a $g\times r$ matrix of indeterminates and $[a_1,\ldots ,a_g]^T=Y[\alpha_1,\ldots ,\alpha_r]^T$, then $a_1,\ldots ,a_g$ form an $R[Y]$-regular sequence such that the squarefree monomials $\ini_{<_1}(a_1),\ldots ,\ini_{<_1}(a_g)$ also form an $R[Y]$-regular sequence.
\end{lem}

The following proposition is a consequence of a more general construction which provides a free resolution---not minimal in general---of the link; we sketch a proof for the convenience of the reader.

\begin{prop}\label{prop:betti}
    Let $d,g$ be positive integers. Let $I$ be a homogeneous ideal of a polynomial ring $S$ of height $g$ such that $S/I$ is a Cohen--Macaulay ring and $I$ has a $d$-linear resolution. Let $a_1,\ldots ,a_g$ be an $S$-regular sequence inside $I$ such that $\deg(a_i)=d+1$ for all $i=1,\ldots ,g$. If $J=(a_1,\ldots ,a_g):I$ then, for all $i\in\{1,\ldots ,g\}$, the graded Betti numbers of $R/J$ are given by the formula:
    $$\beta_{i,j}(S/J)=\begin{cases}b_{g-i+1} & \text{if }j=dg-d+i \mbox{ and }i\neq g-1, \\
    b_2+g & \text{if }j=dg-d+g-1 \mbox{ and }i= g-1, \\
    {\binom{g}{i}} & \text{if }j=di+i \mbox{ and }i<g-1, \\
    0 & \text{otherwise},
    \end{cases}$$
    where $b_s=\displaystyle\prod_{r\in\{1,\ldots ,g\}\setminus\{ s\}}\frac{d+r-1}{|r-s|}$, for $s\in\{1,\ldots ,g\}$, are the Betti numbers of $S/I$.
\end{prop}
\begin{proof}
    Ferrand constructed a free resolution of the link $S/J$ by dualizing and shifting the mapping cone of the natural morphism from the Koszul complex of the complete intersection $S/(a_1,\ldots ,a_g)$ to the minimal free resolution of $S/I$ \cite[Proposition 2.5]{PS74}. The hypotheses that $I$ has a $d$-linear resolution and that $\deg(a_i)>d$ ensure that this free resolution is in fact minimal (apart from an obvious cancellation on the tail). The graded Betti numbers of $S/J$ can then be computed.
\end{proof}

%\begin{rmk}
%    Note that in the situation of Proposition \ref{prop:betti} we have that $\beta_{i,j}(R/J)$ is different from 0 exactly for two integers $j$ whenever $i\leq g-2$, whereas we have that $\beta_{g-1,j}(R/J)=b_2+g$ if $j=dg-d+g-1$ and 0 otherwise and $\beta_{g,j}(R/J)=b_1$ if $j=dg-d+g$ and 0 otherwise. The latter fact is simply because the canonical module of a link of an ideal generated in a single degree is generated in a single degree, while the second depends on the fact that $I$ has linear syzygies and the chosen regular sequence consists of elements of degree $d+1$.
%\end{rmk}
Note that the sequence $a_1,\ldots, a_g$ as in Proposition \ref{prop:betti} exists if $(R,I)$ has the property $\mathbf{P}$, $I$ has a linear resolution and $S$ is the polynomial extension of $R$ which contains the generic link of $I$. The following two results will be useful tools for us in studying the initial ideal of the generic link:

\begin{lem}\label{lem:main}
Let $I\subset R$ be a homogeneous ideal of height $g>0$ such that $R/I$ is a Cohen--Macaulay ring, and
\begin{enumerate}[\quad\rm(1)]
    \item the pair $(R,I)$ has the property $\mathbf{P}$;
    \item $I$ has a $d$-linear resolution.
\end{enumerate}
Then the regular sequence $\underline{\alpha}$ as in $\mathbf{P}$ can be chosen to be part of a minimal generating set of $I$.
If $J$ is the generic link of $I$ (computed by a minimal generating set containing $\underline{\alpha}$), the graded Betti numbers of $J$ and of its initial ideal $\ini_{<_1}(J)$ are equal (position and values). More precisely, for  $i\in\{1,\ldots ,g\}$, we have, 
\[\beta_{i,j}(R[Y]/J)=\beta_{i,j}(R[Y]/\ini_{<_1}(J))=\begin{cases}b_{g-i+1} & \text{if }j=dg-d+i \mbox{ and }i\neq g-1, \\
    b_2+g & \text{if }j=dg-d+g-1 \mbox{ and }i= g-1, \\
    {\binom{g}{i}} & \text{if }j=di+i \mbox{ and }i<g-1, \\
    0 & \text{otherwise},
    \end{cases}\]
    where $b_s=\displaystyle\prod_{r\in\{1,\ldots ,g\}\setminus\{ s\}}\frac{d+r-1}{|r-s|}$, for $s\in\{1,\ldots ,g\}$, are the Betti numbers of $R/I$.
\end{lem}
\begin{proof}
Consider the regular sequence $\underline{\alpha}\subseteq I$ as in $\mathbf{P}$. First, we show that $\underline{\alpha}$ can be chosen to be part of a minimal generating set of $I$. Since $(R,I)$ has the property $\mathbf{P}$, the monomial ideal $\ini_<(I)$ is squarefree by Proposition \ref{p:radical}, so by \cite{CoVa} we have that $\ini_<(I)$ also has a $d$-linear resolution. In particular, $\ini_<(I)$ is generated by squarefree monomials $u_1,\ldots ,u_r$ of degree $d$, which are the initial terms of a minimal generating set $f_1,\ldots ,f_r$ of $I$. Since for all $i\in \{1,\ldots,g\}$, there is a $j_i\in \{1,\ldots,r\}$ such that $u_{j_i}$ divides $\alpha_i$, therefore the polynomials $f_{j_1}, \ldots ,f_{j_g}$ form the regular sequence that we were looking for.

Extend $\underline{\alpha}$ to a minimal generating set $\alpha_1,\ldots ,\alpha_r$ of $I$. If $Y$ is a $g\times r$ matrix of indeterminates and $[a_1,\ldots ,a_g]^T=Y[\alpha_1,\ldots ,\alpha_r]^T$, then $a_1,\ldots ,a_g$ form an $R[Y]$-regular sequence where each $a_i$ has degree $d+1$. Since $J=(a_1,\ldots ,a_g):IR[Y]$, the ring $R[Y]/IR[Y]$ is Cohen--Macaulay, and $IR[Y]$ has a $d$-linear resolution, the formula for the graded Betti numbers of $R[Y]/J$ follows by Proposition \ref{prop:betti}.

By Lemma \ref{lemma:Pinherited}, $\ini_{<_1}(a_1),\ldots ,\ini_{<_1}(a_g)$ is an $R[Y]$-regular sequence of squarefree monomials. So using Lemma \ref{lemma:initialIsLink}, we have 
\[\ini_{<_1}(J)=(\ini_{<_1}(a_1),\ldots ,\ini_{<_1}(a_g)):\ini_{<_1}(IR[Y])\]
and each $\ini_{<_1}(a_i)$ has degree $d+1$. Since $\ini_{<_1}(IR[Y])=\ini_<(I)R[Y]$ is a squarefree monomial ideal, by \cite{CoVa} we get that $R[Y]/\ini_<(I)R[Y]$ is a Cohen--Macaulay ring and that the ideal $\ini_<(I)R[Y]$ has a $d$-linear resolution. The formula for the graded Betti numbers of $R[Y]/\ini_{<_1}(J)$ follows once again from Proposition \ref{prop:betti}.
\end{proof}

\begin{cor}\label{cor:GrobnerBasis}
Under the hypothesis of Lemma \ref{lem:main}, the reduced Gr\"obner basis of the generic link $J$ is a minimal generating set of $J$.    
\end{cor}
\begin{proof}
Note that for any $j\in\mathbb{N},$ the number of homogeneous polynomials of degree $j$ in the reduced Gr\"obner basis of the ideal $J$ (with respect to $<_1$) is given by \[\beta_{1,j}(R[Y]/\ini_{<_1}(J)).\] Since any Gr\"obner basis of $J$ is a set of generators of $J$, the assertion follows immediately from Lemma \ref{lem:main}.    
\end{proof}

%%%%%%%%%%%%%%%%%%%%%%%%%%%%%%%%%%%%%%%%%%%%%%%%%%%%%%%%%%%%%%%%
\section{Description of the initial ideal of the generic link}\label{Section3}
%%%%%%%%%%%%%%%%%%%%%%%%%%%%%%%%%%%%%%%%%%%%%%%%%%%%%%%%%%%%%%%%

Let $I$ be the ideal generated by the maximal minors of a matrix of indeterminates and $J$ its generic link. While the generators of the link $J$ are hard to determine, the aim of this section is to give a precise description of the lead terms of the generators of $J$. We fix the notation for the remainder of the paper first.

\begin{setup}\label{setup}
Let $X$ be an $m \times n$ matrix of indeterminates over a field $K$. Without loss of generality, assume that $1\leq m < n$ (the case $m=n$ is trivial for our purposes, since $I_m(X)$ in this case is a principal ideal, so its generic link is generated by one variable). Let $I = I_m(X)$ be the ideal of $R= K[X]$ generated by the $m$-minors $\{\Delta_{1}, \dots , \Delta_{\eta}\}$ of $X$, where $\eta=\binom{n}{m}$. Let $Y$ be a matrix of indeterminates of size $(n-m+1) \times \eta$ and $S= R[Y]$. Let $\mathfrak{a} \subset IS$ be the ideal generated by the entries of the matrix $Y[\Delta_{1} \dots \Delta_{\eta}]^T$ and $J = \mathfrak{a}S:IS$ be the generic link of $I$. 

Fix any antidiagonal order $<$ on $R$, e.g. the reverse lexicographical order extending the linear order of the variables
\[x_{m,n}>x_{m,n-1}>\cdots>x_{m,1}>x_{m-1,n}>\cdots>x_{2,1}>x_{1,1},\]
and the term order on $S$ given by the property $\mathbf{P}$ for the pair $(S,J)$ as in Lemma \ref{lemma:Pinherited}. It is a classical result, which, to the best of our knowledge, goes back to Narasimhan \cite{Na86}, that
\[\ini_<(I)=(\prod_{i=1}^{m}x_{m-i+1,k_i}:1\leq k_1<k_2<\ldots <k_m\leq n).\]
Fix the initial ideals 
\[  
\ini(IS) \colonequals \ini_{<_1}(IS) = \ini_{<}(I)S  \quad \text{and} \quad \ini(J)\colonequals \ini_{<_1}(J).\]

We now setup the notation to describe the generators of $\ini(J)$. For any $\ell \in \NN$, the symbol $[\ell]$ denotes the set $\{1,2,\ldots,\ell\}$. Define the following sets (see Figure \ref{Fig1} for an illustrative example):
\begin{itemize}
    \item $V\colonequals\{(i,j)\in [m]\times [n]:m-i<j\leq n-i+1\}$;
    \item Fix a subset $A\colonequals\{a_1<a_2<\ldots <a_{m-1}\}\subset \{2,\ldots ,n\}$. Set $a_0\colonequals 1$ and $a_{m}\colonequals n+1$, and let
    \[D_A:=\{(i,j)\in V:a_{m-i}\leq j<a_{m-i+1} \ \forall \ i=1,\ldots ,m\}\subset V.\]
\end{itemize}
\begin{figure}[htbp] 
\begin{center}
\includegraphics[scale=0.22]{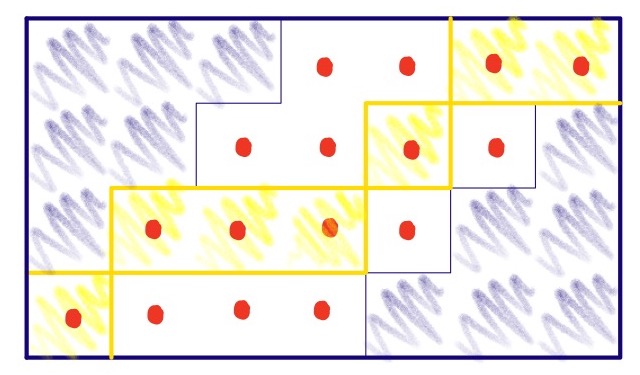} 
\caption{In this picture, and those following, $(1,1)$ is on the upper-left corner, while $(m,n)$ is on the lower-right one. Here, $m=4, \ n=7$, and $\ A=\{2,5,6\}$. The red dots form the set $V$. The subset $D_A\subset V$ consists of the red dots in the yellow area. So, $\beta_A$ is the product of the variables with red dots in the white area.}
\label{Fig1}
\end{center}
\end{figure}
Consider the following squarefree monomials of $R$:
\begin{gather*}
     \alpha_j\colonequals\prod_{i=1}^{m}x_{m-i+1,j+i-1} \quad \text{for}\quad j=1,\ldots ,n-m+1,\\
     \beta_A\colonequals\prod_{(i,j)\in V\setminus D_A}x_{i,j} \quad \text{for every subset} \quad A\subset \{2,\ldots ,n\} \quad \text{of cardinality}\quad m-1.
\end{gather*}
Note that the set of monomials $\{\alpha_j\}$ are pairwise disjoint and precisely the lead terms of the set of adjacent maximal minors $\{[1,m], [2,m+1],\ldots, [n-m+1,n]\}$. By Remark \ref{remark:ForP}, both these sets form a regular sequence. As we will be primarily working with these monomials in the next two sections (instead of the minors), for convenience, we have labeled these monomials as $\alpha_j$ as a slight departure from the notation in Definition \ref{defn:PropertyP}.  
\end{setup}

\begin{rmk}\label{r:propbeta_A}
    Note that $\beta_A\notin \ini(I)$ for any $A$, i.e., for any \[A=\{a_1<a_2<\ldots <a_{m-1}\}\subset \{2,\ldots ,n\}.\] Indeed, if a generator $\prod_{i=1}^{m}x_{m-i+1,k_i}$ of $\ini(I)$ were to divide the monomial $\beta_A$, then we have $(m-i+1,k_i)\notin D_A \ \forall \ i\in [m]$. Since $1\leq k_1<k_2<\ldots <k_m\leq n$, for $i=1$, this means $k_1\geq a_1$, so for $i=2$, we must have $k_2\geq a_2$ (as $k_2>k_1$, so $k_2<a_1$ is not possible). Repeating this argument several times, for $i=m$ we must have $k_m\geq a_{m}=n+1$, which is a contradiction.
\end{rmk}
%Before describing the minimal generators of $\ini(J)$, we make an elementary observation.
%\begin{lem}\label{lem:BettiEqual}
%Let $K\subset L$ be homogeneous ideals in a polynomial ring $S$ over a field. If $\beta_{1,j}(S/K)=\beta_{1,j}(S/L)$ for all $j\in\mathbb{N}$, then $K=L$.
%\end{lem}
%\begin{proof}
%Let $\{f_{1,j},\ldots ,f_{b_j,j}:j\in\mathbb{N}\}$ be a minimal generating set of $K$, where $f_{i,j}$ has degree $j$ and $b_j=\beta_{1,j}(S/K)$. By contradiction, let $g$ be a homogeneous element of $L\setminus K$ having minimal degree, say $d$. Then some $f_{i,d}$ must lie in the ideal generated by the set $\{f_{1,j},\ldots ,f_{b_j,j}:j\leq d\}\setminus \{f_{i,d}\}\cup\{g\}$, otherwise $\beta_{1,d}(S/L)>\beta_{1,d}(S/K)$. In writing the corresponding equation, $g$ must occur with a nonzero coefficient of the base field. Since such a coefficient is invertible, the same equation tells us that $g\in (f_{1,j},\ldots ,f_{b_j,j}:j\leq d)\subset K$, a contradiction.
%\end{proof}

Before describing the minimal generators of $\ini(J)$, we note the following:
\begin{rmk}\label{rem:in(J)}
Since $(R,I)$ has the property $\mathbf{P}$, $R/I$ is a Cohen--Macaulay ring, and $I$ has an $m$-linear resolution, so by Corollary \ref{cor:GrobnerBasis} we have that the reduced Gr\"obner basis of $J$ is a minimal set of generators of $J$. By Lemma \ref{lem:main}, this Gr\"obner basis consists of $n-m+1$ homogeneous polynomials of degree $m+1$ (which are precisely the entries of the matrix $Y[\Delta_1,\ldots ,\Delta_r]^T$) and $\binom{n-1}{m-1}$ homogeneous polynomials each of degree $m(n-m)+1$ .
Note that, if $n=m+1$, this translates into the fact that $J$ has $m+2$ minimal homogeneous generators of degree $m+1$.
\end{rmk}

\begin{thm}\label{thm:in(J)}
The minimal generators of the initial ideal $\ini(J)$, equivalently the lead terms of the reduced Gr\"obner basis of $J$, are given as follows:
    \begin{enumerate}[\quad\rm(1)]
        \item $n-m+1$ squarefree monomials of $S$, each of degree $m+1$, given by \[Y_{j,j}\cdot \alpha_j \quad  \text{for} \quad j=1,\ldots ,n-m+1;\]
        \item $\binom{n-1}{m-1}$ squarefree monomials of $S$, each of degree $m(n-m)+1$, given by \[\prod_{j=1}^{n-m+1}Y_{j,j}\cdot \beta_A,\] for every subset $A\subset \{2,\ldots ,n\}$ of cardinality $m-1$.
    \end{enumerate}
\end{thm}
The generators listed in item (1) are precisely the lead terms of the regular sequence defining the generic link $J$. The nontrivial generators of $\ini(J)$---those listed in item (2)---will play a crucial role in studying the symbolic powers of $J$ in the next section. Before proving the theorem, we examine it in some special cases.
\begin{rmk}\label{rmk:lead terms}
The precise generators of the generic link $J$ are known for the case $m=1$, i.e., for the generic link of the homogeneous maximal ideal.  Let $R=K[x_1,\ldots,x_n]$ and $I=(x_1,\ldots,x_n)$ be the homogeneous maximal ideal. The generators of the generic link of $I$ were calculated by Huneke and Ulrich in \cite[Example 3.4]{Huneke-Ulrich88}: Let $Y\colonequals (Y_{i,j})$ be an $n \times n$ matrix of indeterminates and $S=R[Y]$. Then
\[J = Y[x_1, \cdots, x_n]^TS+\det(Y)S.\] 
Note that, for this case, Theorem \ref{thm:in(J)} says that under the monomial ordering as in Setup \ref{setup}, 
the above generators of $J$ form a Gr\"obner basis. Indeed, the only choice for the set $A$ as in Theorem \ref{thm:in(J)} (2) is the empty set $\phi$, and $\beta_\phi = 1$ so that by Theorem \ref{thm:in(J)},
\[\ini(J) = (Y_{1,1}x_1,\;\ldots,\; Y_{n,n}x_n,\;Y_{1,1}Y_{2,2}\ldots Y_{n,n}).\]
In general, the generators of $\ini(J)$ arise from intriguing combinatorial patterns. To see this, let $m=3$ and $n=5$. The possible choices for $A$ are
\[ \{2,3\},\;\{2,4\},\;\{2,5\},\;\{3,4\},\;\{3,5\},\; \text{and}\; \{4,5\}.\] 
The nontrivial generators of $\ini(J)$ (divided by the common factor $\prod_{j=1}^{n-m+1}Y_{j,j}$) are depicted in Figure \ref{Fig2}. 
\begin{figure}[htbp]
\begin{center}
\includegraphics[scale=0.22]{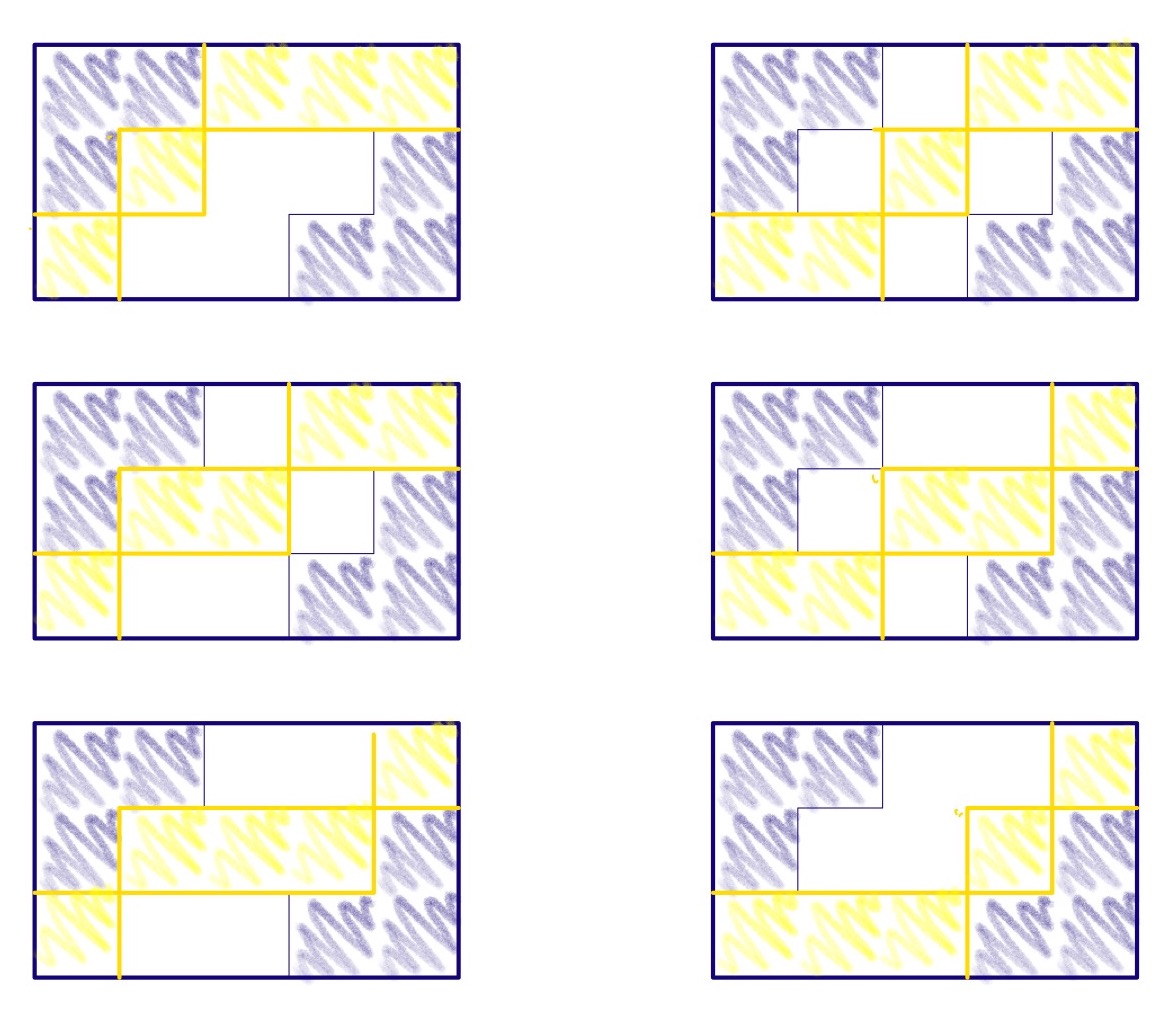} 
\caption{The six monomials $\beta_A$ are the products of the variables in the non-shaded area of the respective pictures. For instance, in the first picture, $A=\{2,3\}$ and $\beta_A=x_{2,3}x_{2,4}x_{3,2}x_{3,3}$. The monomials $\alpha_1,\alpha_2,$ and $\alpha_3$, are the three main anti-diagonals of the matrix. For instance, $\alpha_1=x_{1,3}x_{2,2}x_{3,1}$.}
\label{Fig2}
\end{center}
\end{figure}
\end{rmk}
\begin{proof}
By Remark \ref{r:propbeta_A} no monomial in (1) divides any monomial in (2). Other divisibility relations between monomials in (1) and (2) are of course not  possible for degree reasons. Since the degrees and the number of monomials asserted in Theorem \ref{thm:in(J)} agrees with those in Remark \ref{rem:in(J)}, in view of the fact that 
$\ini(J) = \ini(a):\ini(I)$ (Lemma \ref{lemma:initialIsLink}), it only remains to show that 
\[\ini(I)=(\prod_{i=1}^{m}x_{m-i+1,k_i}:1\leq k_1<k_2<\ldots <k_m\leq n)\] 
multiplies all the monomials listed in Theorem \ref{thm:in(J)} into \[\ini(\mathfrak{a})=(Y_{j,j}\cdot \alpha_j:j=1,\ldots ,n-m+1).\] This is clear for the monomials of type $(1)$. So we have to show that 
\[\ini(I)\cdot \prod_{j=1}^{n-m+1}Y_{j,j}\cdot \beta_A\in\ini(\mathfrak{a}) \quad \forall A\subset \{2,\ldots ,n\} \; \text{with}\; |A|=m-1. \] 
Pick a generator $\mu \colonequals \prod_{i=1}^{m}x_{m-i+1,k_i}$ of $\ini(I)$ and a subset $A=\{a_1,\ldots,a_{m-1}\} \subset \{2,\ldots,n\}$. It suffices to prove that:
\[\exists j\in[n-m+1] \quad \text{such that} \quad \alpha_j=\prod_{i=1}^{m}x_{m-i+1,j+i-1}|\mu \beta_A.\]
In order to prove the aforementioned, set
\[\epsilon \colonequals \min\{i \in [m] : k_i+1<a_{i}\} \quad \text{with} \; a_{m}\colonequals \infty,\quad \text{and}\;  j\colonequals k_\epsilon -\epsilon+1,\]
and note that $1\leq j\leq n-m+1$ since $\epsilon \leq k_\epsilon \leq n-(m-\epsilon)$. With this particular choice of $j$, we claim that $\alpha_j|\mu\beta_A$. To prove it, we need to show that
\[\forall i\in [m], \quad \text{either} \quad (m-i+1,j+i-1) \notin D_A \quad \text{or} \quad j+i-1=k_i. \]
Indeed, this can be verified for each case:
\begin{itemize}
    \item If $i=\epsilon$, then $j+\epsilon-1=k_\epsilon$ by the definition of $j$. 
    \item If $i>\epsilon$, then
    \[j+i-1=j+\epsilon-1+(i-\epsilon)=k_\epsilon+(i-\epsilon)<a_{\epsilon}+(i-\epsilon-1)\leq a_{\epsilon+i-\epsilon-1}=a_{i-1},\]
    so $(m-i+1,j+i-1) \notin D_A$.
    \item If $i<\epsilon$, suppose that $(m-i+1,j+i-1) \in D_A$, then $a_{i-1}\leq j+i-1<a_{i}$. Note that $k_i=j+i-1$: Indeed, $k_i \leq j+i-1=k_\epsilon-(\epsilon-i)$ is clear.
    \[\text{If} \quad k_i<j+i-1,\quad \text{then} \quad k_i+1<j+i\leq a_i,\]
    contradicting the minimality of $\epsilon$. 
\end{itemize}
This concludes the proof of the theorem.
\end{proof}

%%%%%%%%%%%%%%%%%%%%%%%%%%%%%%%%%%%%%%%%%%%%%%%%%%%%%%%%%%%%%%%%%%%%%%%%
\section{The symbolic and ordinary powers of the generic link\\ are equal}\label{Section4}
%%%%%%%%%%%%%%%%%%%%%%%%%%%%%%%%%%%%%%%%%%%%%%%%%%%%%%%%%%%%%%%%%%%%%%%%

To begin with, we recall the definition of symbolic powers. Let $\mathfrak{b}$ be a radical ideal in a Noetherian ring $A$. For a fixed $\ell \in \NN$, the $\ell$th symbolic power of $\mathfrak{b}$ is
\[\mathfrak{b}^{(\ell)}\colonequals \bigcap_{\frakp \in \Ass(\mathfrak{b})} \frakp^\ell A_{\frakp}\cap A. \]
It turns out that the equality $\mathfrak{b}^{(\ell)}=\mathfrak{b}^\ell$ says that the ideal $\mathfrak{b}^\ell$ has no embedded primes.

We remind the reader that we stick to the notation of Setup \ref{setup}. The aim of this section is to establish the equality $J^{(\ell)}=J^\ell$ of the symbolic and ordinary powers of the generic link for all $\ell \in \NN$. The following general observation will be quite useful for us (see \cite[Chapter 7]{BCRV} for the background on $F$-singularities):

\begin{prop}\label{prop:crucialmethod}
Let $U$ be a proper ideal in a polynomial ring over a field. If there exists a term order $\prec$ such that $\ini_{\prec}(U)$ is a squarefree monomial ideal and $\ini_{\prec}(U)^{(\ell)}=\ini_{\prec}(U)^\ell$ for all $\ell \in \NN$, then:
\begin{enumerate}[\quad\rm(1)]
\item $U^{(\ell)}=U^\ell$ for all $\ell \in \NN$.
\item $\ini_{\prec}(U^\ell)=\ini_{\prec}(U)^\ell$ for all $\ell \in \NN$.
\item The Rees algebra $\mathcal{R}(U)=\oplus_{\ell\in\NN}U^\ell$ has rational singularities in characteristic 0 and is $F$-rational in positive characteristic.
\item If $U$ is unmixed, it has the property $\mathbf{P}$.
\item Over an $F$-finite field of positive characteristic, if $U$ is unmixed then both $\mathcal{R}(U)$ and the associated graded ring $\gr(U)=\oplus_{\ell\in\NN}U^\ell/U^{\ell+1}$ are $F$-split.
\end{enumerate}
\end{prop}
\begin{proof}
It is harmless to assume that the field is perfect. For each $\ell \in \NN$, we have the chain of containments
\[\label{eq:containments}
\ini_{\prec}(U)^\ell \subset \ini_{\prec}(U^\ell) \subset \ini_{\prec}(U^{(\ell)}) \subset \ini_{\prec}(U)^{(\ell)}, 
\]
where the last containment is due to \cite[Proposition 5.1]{Sullivant} because $\ini_{\prec}(U)$ is squarefree. Since $\ini_{\prec}(U)^\ell=\ini_{\prec}(U)^{(\ell)}$ by hypothesis, each containment in the above display is an equality and we get (1) and (2).

As for (3), note that from (2) we have that the initial subalgebra of the Rees algebra is the Rees algebra of the initial ideal, i.e. $\ini_{\prec}(\mathcal{R}(U))=\mathcal{R}(\ini_{\prec}(U))$. The latter is a finitely generated normal domain since $\ini_{\prec}(U)^{(\ell)}=\ini_{\prec}(U)^\ell$ for all $\ell \in \NN$, so (3) follows from \cite[Corollary 7.3.13]{BCRV}.

To prove (4), say that $U$ is unmixed of height $g>0$. Since $\ini_{\prec}(U)$ is squarefree, it is unmixed of height $g$ as well by \cite[Corollary 2.1.8]{BCRV}. So any associated prime ideal of $\ini_{\prec}(U)$ is generated by $g$ variables, and therefore, if $\nu$ is the product of all the variables of the ambient polynomial ring, $\nu\in \mathfrak{p}^g=\mathfrak{p}^{(g)} \ \forall \ \mathfrak{p}\in\mathrm{Ass}(\ini_{\prec}(U))$.  Hence $\nu$ is a squarefree monomial contained in
\[\ini_{\prec}(U^{(g)})=\ini_{\prec}(U^g)=\ini_{\prec}(U)^g=\ini_{\prec}(U)^{(g)}=\bigcap_{\mathfrak{p}\in \mathrm{Ass}(\ini_{\prec}(U))}\mathfrak{p}^g.\]
Exploiting the above chain of equalities, there exist monomials $u_1,\ldots ,u_g\in\ini_{\prec}(U)$ such that $\prod_{i=1}^gu_i|\nu$. Therefore $u_i$ and $u_j$ must be coprime for all $i\neq j$. For all $i\in[g]$, choose $\alpha_i\in U$ such that $\ini_{\prec}(\alpha_i)=u_i$, then $\underline{\alpha}=\alpha_1,\ldots ,\alpha_g$ is the regular sequence in $U$ as in the property $\mathbf{P}$.

To prove (5), note that as shown in item (4), we have $\nu\in\ini(U^{(g)})$. Then $U$ is a \textit{symbolic $F$-split} ideal by \cite[Definition 5.2, Lemma 6.2]{dS-Mn-NB}, i.e., $\{U^{(\ell)}=U^\ell\}_{\ell\in\NN}$ is an \textit{$F$-split filtration}. That is, both $\mathcal{R}(U)$ and $\gr(U)$ are $F$-split rings by \cite[Theorem 4.7]{dS-Mn-NB}.
\end{proof}

In view of Proposition \ref{prop:crucialmethod}, our goal is to prove the equality of symbolic and ordinary powers of the initial ideal $\ini(J)$. In order to do this, we will need the description of $\ini(J)$ obtained in Theorem \ref{thm:in(J)}. Set
\[N:=(\beta_A:A\subset \{2,\ldots ,n\},  \ |A|=m-1) \subset S.\]
Recall that by Theorem \ref{thm:in(J)},
$$\ini(J)=(Y_{1,1}\alpha_1,\ldots ,Y_{n-m+1,n-m+1}\alpha_{n-m+1})+\prod_{j=1}^{n-m+1}Y_{j,j}\cdot N.$$

One difficulty in proving that the ordinary and symbolic powers of $\ini(J)$ agree lies in the fact that the analogous statement for $N$ is not true in general:  

\begin{exa}\label{exU}
Assume $m>2$ and $n>m+1$. We claim that $N^{(2)}\neq N^2$. Note that if $m>1$ and $n>m$, then the ideal $N$ has height $2$ as $N\subset (x_{m-1,2},x_{m,2})$. Let $\nu$ be the product of all the variables of $S$. Since $N$ is squarefree, we have \[\nu \in N^{(2)} = \bigcap_{\mathfrak{p}\in\mathrm{Min}(N)}\mathfrak{p}^{2} .\] However, for any subsets $A$ and $B$ of $\{2,\ldots ,n\}$ of cardinality $m-1$, $\beta_A\beta_B$ is \textit{not} squarefree if $m>2$ and $n>m+1$. For instance, at least one among the variables $x_{m-2,3},x_{m-1,3}$, or $x_{m,3}$ must divide both $\beta_A$ and $\beta_B$ so that $\nu\notin N^2$. In fact, in Corollary \ref{cor:NontrivialGenerators}, we will show that $N^{(\ell)}=N^\ell$ for all $\ell$ if and only if $m\leq 2$ or $n=m+1$.
\end{exa}
    
In order to establish the equality $\ini(J)^{(\ell)}=\ini(J)^{\ell}$ of the symbolic and ordinary powers of the initial ideal, we will make use of the following criterion of Monta\~no and N\'u\~nez Betancourt:

\begin{thm}\label{t:useful}\cite[Corollary 4.10]{Jonathon-Luis}
Let $W\subset S$ be a squarefree monomial ideal. Then $W^{(\ell)}=W^{\ell}$ for all $\ell\in\NN$ if and only if the ideal \[(W^{r+1})^{[2]}:W^{2r+1}\] contains the product $\nu$ of all the variables of $S$ for any $r\in\NN$ with $r<\mu(W)/2$. Here $(W^{r+1})^{[2]}$ is the monomial ideal generated by all the squares of the monomials in $W^{r+1}$.
\end{thm}
    
In what follows (Remark \ref{rem:exp} -- Lemma \ref{lem:mostcases}), we make a number of careful observations in order to put ourself in a situation where we can apply Theorem \ref{t:useful}. 

\begin{rmk}\label{rem:exp}
Let $A=\{a_1<\ldots <a_{m-1}\}$, set $a_0=1$ and $a_{m}=n+1$. Let $(i,j) \in V$ and $ e_{x_{i,j}}(\beta_A)$ be the highest power of $x_{i,j}$ which divides $\beta_A$. Then we have:
    $$
    e_{x_{i,j}}(\beta_A)=\begin{cases}
        0 & \mbox{ if } \ (i,j)\in D_A, \\
        1 & \mbox{ otherwise}.
    \end{cases}$$
Equivalently,
    $$
    e_{x_{i,j}}(\beta_A)=\begin{cases}
        0 & \mbox{ if } \ a_{m-i}\leq j<a_{m-i+1}, \\
        1 & \mbox{ otherwise}.
    \end{cases}$$
The variables dividing $\beta_A$ split into two groups (see Figure \ref{Fig3}). To make this precise, for $(i,j)\in V\setminus D_A$, we say that \textit{$(i,j)$ is to the right of $D_A$} if $j\geq a_{m-i+1}$, while\textit{ $(i,j)$ is to the left of $D_A$} if $j<a_{m-i}$. Clearly, if $(i,j)\in V\setminus D_A$, then $(i,j)$ is either to the right or to the left of $D_A$.

\begin{figure}[htbp] 
\begin{center}
\includegraphics[scale=0.22]{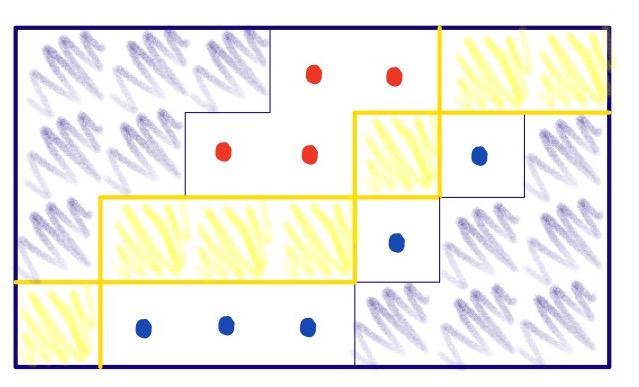} 
\caption{Here $m=4, \ n=7$, and $A=\{2,5,6\}$. The red dots are to the left of $D_A$, while the blue dots are to the right of $D_A$.}
\label{Fig3}
\end{center}
\end{figure}
\end{rmk}

Given $A=\{a_1<\ldots <a_{m-1}\},\; B=\{b_1<\ldots <b_{m-1}\}\subset\{2,\ldots ,n\}$, define 

\begin{align*}
A\wedge B\colonequals\{\min\{a_1,b_1\},\ldots ,\min\{a_{m-1},b_{m-1}\}\}, \\
A\vee B\colonequals\{\max\{a_1,b_1\},\ldots ,\max\{a_{m-1},b_{m-1}\}\}.
\end{align*}

The following observation introduces a ``straightening law" on the ideal $N$:
\begin{lem}\label{lem:straighten}
Let $A$ and $B$ are subsets of $\{2,\ldots ,n\}$ of cardinality $m-1$. Then we have \[\beta_A\cdot \beta_B=\beta_{A\wedge B}\cdot \beta_{A\vee B}.\]
\end{lem}
\begin{proof}
It suffices to check that $e_{x_{i,j}}(\beta_A\cdot \beta_B)=e_{x_{i,j}}(\beta_{A\wedge B}\cdot \beta_{A\vee B})$ for each $(i,j)\in V$. This is straightforward by Remark \ref{rem:exp}. 
\end{proof}

Given subsets $A=\{a_1<\ldots <a_{m-1}\}$ and $B=\{b_1<\ldots <b_{m-1}\}$ of $\{2,\ldots ,n\}$, we say that $A\leq B$ if $a_i\leq b_i$ for all $i=1,\ldots ,m-1$. This induces the structure of a \textit{distributive lattice} on the collection of subsets of $\{2,\ldots ,n\}$ of cardinality $m-1$. As an immediate consequence of Lemma \ref{lem:straighten}, we get a more manageable description of the powers of $N$. 

\begin{cor}\label{c:Upowers}
For all $s\in\NN$, we have
$$N^s=(\prod_{k=1}^s\beta_{A_{k}}:A_1\leq \ldots \leq A_{s} \mbox{ are subsets of $\{2,\ldots ,n\}$ of cardinality $m-1$}).$$
\end{cor}

\begin{rmk}\label{rem:key}
Let $A$ and $B$ be subsets of $\{2,\ldots ,n\}$ of cardinality $m-1$ with $A\leq B$. For any $(i,j)\in V$, we have:
 \begin{itemize}
    \item If $(i,j)$ is to the right of $D_B$, then it is also to the right of $D_A$.
    \item If $(i,j)$ is to the left of $D_A$, then it is also to the left of $D_B$.
    \end{itemize}
    Therefore if $A_1\leq \ldots \leq A_{\ell}$ are subsets of $\{2,\ldots ,n\}$ of cardinality $m-1$ and $\gamma=\prod_{k=1}^{\ell}\beta_{A_{k}}$, then for all $(i,j)\in V$, we have:
    \begin{align}\label{eq:key}
    e_{x_{i,j}}(\gamma) \ \ = & \ \ \max\{k\in [\ell]:(i,j) \mbox{ is to the right of } D_{A_k}\}+ \nonumber \\
    & \ell -\min\{k\in[\ell]:((i,j) \mbox{ is to the left of } D_{A_k}\}+1,
    \end{align}
where $\max\{k\in [\ell]:(i,j) \mbox{ is to the right of } D_{A_k}\}=0$ if $(i,j)$ is not to the right of $D_{A_k}$ for any $k\in[\ell]$ and $\min\{k\in[\ell]:(i,j) \mbox{ is to the left of } D_{A_k}\}=\ell+1$ if $(i,j)$ is not to the left of $D_{A_k}$ for any $k\in[\ell]$.    
\end{rmk}

\begin{lem}\label{lem:mostcases}
For any positive integer $r$ and $A_1\leq \ldots \leq A_{2r+1}$ subsets of $\{2,\ldots ,n\}$, we have
$$\big(\prod_{k=1}^r\beta_{A_{2k}}\big)^2\big|\prod_{h=1}^{2r+1}\beta_{A_h}.$$ 
In particular, $N^s\subset (N^r)^{[2]}$ whenever $s\geq 2r+1$.
\end{lem}
\begin{proof}
    Let $(i,j)\in V$. By Equation \ref{eq:key}, we have
    \begin{align*}
   e_{x_{i,j}}\big(\prod_{k=1}^r\beta_{A_{2k}}\big) \ = & \ \overbrace{\max\{k\in[r]:(i,j) \mbox{ is to the right of } D_{A_{2k}}\}}^{\mathrm{max}_e}  \ +\\
   & r -\underbrace{\min\{k\in[r]:((i,j) \mbox{ is to the left of } D_{A_{2k}}\}}_{\mathrm{min}_e}+1, \\
 e_{x_{i,j}}\big(\prod_{h=1}^{2r+1}\beta_{A_{h}}\big) \ = & \ \overbrace{\max\{h\in[2r+1]:((i,j) \mbox{ is to the right of } D_{A_{h}}\}}^{\mathrm{max}_o} \ + \\
    & 2r+1 -\underbrace{\min\{h\in[2r+1]:(i,j) \mbox{ is to the left of } D_{A_{h}}\}}_{\mathrm{min}_o}+1.
   \end{align*}
 
    Notice that $2\cdot \mathrm{min}_e\geq \mathrm{min}_o$ and $2\cdot \mathrm{max}_e\leq \mathrm{max}_o$, hence we infer from the above equations that $2\cdot e_{x_{i,j}}\big(\prod_{k=1}^r\beta_{A_{2k}}\big)\leq e_{x_{i,j}}\big(\prod_{k=1}^{2r+1}\beta_{A_{k}}\big)$
    for all $(i,j)\in V$. This proves the first assertion. The second assertion follows at once from Corollary \ref{c:Upowers}.
\end{proof}

We are now ready to prove the key technical result of this section.

\begin{thm}\label{thm:main}
    We have $\ini(J)^{(\ell)}=\ini(J)^{\ell}$ for all $\ell\in\mathbb{N}$.
    \end{thm}
\begin{proof}
Let $\nu$ be the product of all the variables of $S$. By Theorem \ref{t:useful}, our aim is to prove $$\nu\in(\ini(J)^{r+1})^{[2]}:\ini(J)^{2r+1} \ \forall \ r\in\mathbb{N}.$$ 
First we will show that $\nu$ multiplies the product of $2r+1$ distinct generators of $\ini(J)$ into $(\ini(J)^{r+1})^{[2]}$, that as it will be clear from the explanation at the end of the proof it is the most complicated part.

Set $\mu:=Y_{1,1}Y_{2,2}\cdots Y_{n-m+1,n-m+1}\in S$; pick natural numbers $a$ and $b$ such that $a+b=2r+1$,
 $1\leq i_1<i_2<\ldots <i_a\leq n-m+1$ and subsets $A_1< A_2< \ldots < A_b$ of $\{2,\ldots ,n\}$ each of cardinality $m-1$.
We claim that
\[\gamma:=\nu \cdot \prod_{k=1}^aY_{i_k,i_k}\alpha_{i_k}\cdot \prod_{k=1}^b\mu\beta_{A_k}\in \big(\ini(J)^{r+1}\big)^{[2]}.\]
    
 \medskip
 
\noindent {\bf If $a\geq 2$:} Notice that, if $\delta:=\prod_{k=1}^aY_{i_k,i_k}\alpha_{i_k}\cdot \prod_{k=1}^{\lfloor\frac{b-1}{2}\rfloor}\mu\beta_{A_{2k}}$, then $\delta^2|\gamma$. Indeed, the square of $\prod_{k=1}^aY_{i_k,i_k}\alpha_{i_k}$ clearly divides $\nu \cdot \prod_{k=1}^aY_{i_k,i_k}\alpha_{i_k}$, while $\big(\prod_{k=1}^{\lfloor\frac{b-1}{2}\rfloor}\mu\beta_{A_{2k}}\big)^2$ divides $\prod_{k=1}^b\mu\beta_{A_k}$ by Lemma \ref{lem:mostcases}. So it remains to show that $\delta\in\ini(J)^{r+1}$. This follows immediately from the inequality
\[a+\left\lfloor\frac{b-1}{2}\right\rfloor\geq r+1,\]
which is a routine verification.    

\smallskip

\noindent {\bf If $a=1$:} In this case, we have $b=2r$. Define \[A_{b+1}\colonequals\{i_1+1,i_1+2,\ldots , i_1+m-1\}\subset \{2,\ldots ,n\},\] and note that $\beta_{A_{b+1}}$ divides $\nu/(Y_{i_1,i_1}\alpha_{i_1})$. By Lemma \ref{lem:mostcases}, $\prod_{k=1}^{b+1}\beta_{A_k}\in \big(N^r\big)^{[2]}$, so we have $\frac{\nu}{Y_{i_1,i_1}\alpha_{i_1}}\cdot \prod_{k=1}^{b}\beta_{A_k}$ lies in $\big(N^r\big)^{[2]}$ as well, thus
    $$
    \frac{\nu}{Y_{i_1,i_1}\alpha_{i_1}}\cdot \prod_{k=1}^{b}\mu\beta_{A_k}\in \big(\mu^r N^r\big)^{[2]}\subset \big(\ini(J)^r\big)^{[2]}.
    $$
 Hence, $\gamma=\nu \cdot Y_{i_1,i_1}\alpha_{i_1}\cdot \prod_{k=1}^b\mu\beta_{A_k}\in (Y_{i_1,i_1}\alpha_{i_1})^2\big(\ini(J)^r\big)^{[2]}\subset \big(\ini(J)^{r+1}\big)^{[2]}$. 

\smallskip

\noindent {\bf If $a=0$:} This case turns out to be the most delicate to deal with. Note that we have $b=2r+1$. Consider the monomial $\delta:=\prod_{k=1}^{r+1}\mu\beta_{A_{2k-1}}\in\ini(J)^{r+1}$. 
 \begin{itemize}
 \item[(i)] Suppose that for all $(i,j)\in V$, 
 $$(i,j)\notin D_{A_1}\cup D_{A_{2r+1}}\implies (i,j)\notin D_{A_k} \ \forall \ k\in [2r+1].$$ 
 For $(i,j)\in V$, if $e_{x_{i,j}}(\delta)>0$, then either $(i,j)$ is to the right of $D_{A_1}$ (and in this case $e_{x_{i,j}}(\delta)=\max\{k\in[r+1]:(i,j)\notin D_{A_{2k-1}}\}$ while $e_{x_{i,j}}(\gamma)=\max\{h\in[2r+1]:(i,j)\notin D_{A_h}\}+1$, where the ``+1" comes from the contribution of $\nu$), or $(i,j)$ is to the left of $D_{A_{2r+1}}$ (and in this case $e_{x_{i,j}}(\delta)=r+1-\min\{k\in[r+1]:(i,j)\notin D_{A_{2k-1}}\}+1$ while $e_{x_{i,j}}(\gamma)=2r+1-\min\{h\in[2r+1]:(i,j)\notin D_{A_h}\}+1+1$, where one of the ``+1" comes from the contribution of $\nu$).  In both cases, $2e_{x_{i,j}}(\delta)\leq e_{x_{i,j}}(\gamma)$, and so $\delta^2|\gamma$.
 
 \item[(ii)] Suppose that there exists $(i,j)\in V$ and $k\in\{2,\ldots ,2r\}$ such that 
 $$(i,j)\in D_{A_k}\setminus (D_{A_1}\cup D_{A_{2r+1}}).$$ 
(The reader may go over the following construction while keeping an eye on Figure \ref{Fig4}). Notice that we must have $1<i<m$ and $m-i+1<j<n-i+1$. Furthermore, if $p+q=i+j$, then $(p,q)$ is to the right of $D_{A_1}$ whenever $p\geq i$ and $(p,q)$ is to the left of $D_{A_{2r+1}}$ whenever $p\leq i$. So $\alpha_{i+j-m}|\beta_{A_1}\beta_{A_{2r+1}}$.
 Now write $A_1=\{a_1<\ldots <a_{m-1}\}$ and $A_{2r+1}=\{b_1<\ldots <b_{m-1}\}$. Let 
 \begin{align*}
 \omega \colonequals\min\{s\in [m-1]:m-s+b_{s}\geq i+j\}, \\
 \Omega \colonequals\max\{s\in [m-1]:m-s+a_{s}\leq i+j\}.
 \end{align*} 
 Notice that $\omega\leq m-i$ and $\Omega>m-i$, and define 
 \[C\colonequals\{c_1,\ldots ,c_{m-1}\} \ \mbox{ where }c_r=\begin{cases}b_r & \mbox{ if } \ r<\omega \\
 i+j+r-m & \mbox{ if } \omega\leq r\leq \Omega \\
 a_r &\mbox{ if } \ r>\Omega.
 \end{cases}\]
  Note that
$(p,i+j-p)\notin D_{A_k}$ for all $k\in[2r+1]$ if $m-p<\omega$ (in which case $(p,i+j-p)$ is on the right of $D_{2r+1}$) or $m-p>\Omega$ (in which case $(p,i+j-p)$ is on the left of $D_{1}$). Hence, if $\epsilon:=\prod_{m-p<\omega}x_{p,i+j-p}\cdot \prod_{m-p>\Omega}x_{p,i+j-p}$, we have that $\epsilon^{2r+2}|\gamma$ and $(\beta_C\cdot \alpha_{i+j-m})^2|\beta_{A_1}\cdot \beta_{A_{2r+1}}\cdot \nu\cdot \epsilon$: indeed the greater common divisor of $\beta_C$ and of $\alpha_{i+j-m}$ is $\epsilon$, and it divides all the four factors. Moreover we already noticed that $\alpha_{i+j-m}|\beta_{A_1}\cdot \beta_{A_{2r+1}}$, hence $(\alpha_{i+j-m})^2$ divides $\beta_{A_1}\cdot \beta_{A_{2r+1}}\cdot \nu$; finally, $\beta_C|\beta_{A_1}\cdot \beta_{A_{2r+1}}$ since $A_1\leq C\leq A_{2r+1}$, hence $(\beta_C)^2|\beta_{A_1}\cdot \beta_{A_{2r+1}}\cdot \nu$.

Next, by Lemma \ref{lem:mostcases}, $(\prod_{h=1}^{r-1}\beta_{A_{2h+1}})^2|\prod_{l=2}^{2r}\beta_{A_l}$, and since $\epsilon$ divides the squarefree monomial $\beta_{A_l}$ for all $l\in [2r+1]$, the square of $\prod_{h=1}^{r-1}\beta_{A_{2h+1}}$ divides $\frac{\prod_{l=2}^{2r}\beta_{A_l}}{\epsilon}$, so $(\beta_C\cdot \alpha_{i+j-m}\cdot \prod_{h=1}^{r-1}\beta_{A_{2h+1}})^2|\gamma$.

\medskip
 
\begin{minipage}{0.30\textwidth}  
\begin{center}
\includegraphics[scale=0.20]{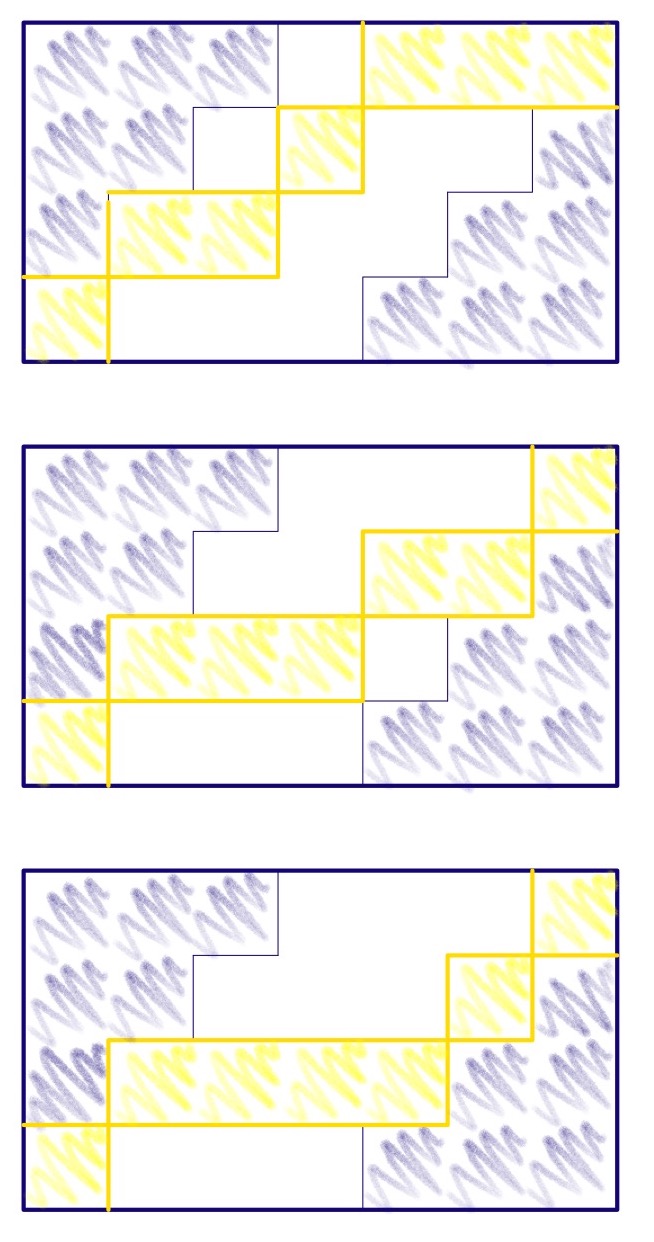} 
\end{center}
\end{minipage}\hfill
\begin{minipage}{0.60\textwidth}  
\begin{center}
\includegraphics[scale=0.20]{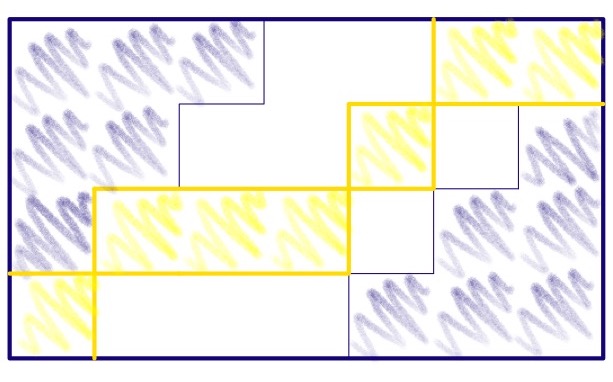} 
\captionof{figure}{In this situation where $a=0$, we have $m=4$, $n=7$, and $r=1$ (so $b=3$). On the left, we have the illustrations of $A_1=\{2,4,5\}$, $A_2=\{2,5,7\}$, and $A_3=\{2,6,7\}$. This situation falls in the case (ii) since $(2,5)\in D_{A_2}\setminus (D_{A_1}\cup D_{A_3})$. Note that $\omega=2$ and $\Omega=3$, so $C=\{2,5,6\}$. The illustration above depicts $D_C$.}
\label{Fig4}
\end{center}
\end{minipage}  
\end{itemize}
 
\noindent Now the theorem follows from Theorem \ref{t:useful}. Using Corollary \ref{c:Upowers}, a generator $\xi$ of $\ini(J)^{2r+1}$ has the form
 $$
  \xi=\prod_{k=1}^{n-m+1}(Y_{i_k,i_k}\alpha_{i_k})^{r_k}\cdot \prod_{j=1}^b(\mu\beta_{A_j})^{s_j},
 $$
 where $r_k$ and $s_j$ are natural numbers whose sum  is $2r+1$, and $A_1<A_2<\ldots<A_b$ are subsets of $\{2,\ldots ,n\}$ of cardinality $m-1$. Consider the sets 
 \[P\colonequals\{k\in[n-m+1]:r_k\mbox{ is odd}\} \quad \text{and} \quad Q\colonequals \{j\in [b]:s_j\mbox{ is odd}\},\] and notice that  $|P|+|Q|$ is odd, say $|P|+|Q|=2s+1$. Consider the monomial
 $$
 \xi'=\prod_{k\in P}Y_{i_k,i_k}\alpha_{i_k}\cdot \prod_{j\in Q}\mu\beta_{A_j}.
 $$
 It is clear that $\xi/\xi'\in \ini(J)^{2(r-s)}$ is the square of a monomial in the ideal $\ini(J)^{r-s}$. On the other hand, by what we just proved, $\nu\cdot \xi'\in (\ini(J)^{s+1})^{[2]}$, so \[\nu\cdot \xi=\nu \cdot \xi'\cdot \xi/\xi'\in (\ini(J)^{s+1})^{[2]}(\ini(J)^{r-s})^{[2]}=     (\ini(J)^{r+1})^{[2]}.\]
 This concludes the proof.
\end{proof}

The main result of this section follows immediately from Theorem \ref{thm:main} in view of Proposition \ref{prop:crucialmethod}.

\begin{cor}\label{cor:OrdinaryEqualsSymbolic}
 The symbolic and ordinary powers of the generic link $J$ agree, i.e.,
 \[J^{(\ell)}=J^\ell \quad \text{for all} \quad \ell\in \mathbb{N}.\]
\end{cor}

Note that the generic link of a homogeneous ideal generated in degree $d$ is not a \textit{minimal link}\footnote{The link $J=\fraka : I$ is called minimal if the generators of the ideal $\fraka$ can be extended to a minimal generating set of $I$. Minimal linkage is generally better behaved and preferable over arbitrary linkage. For example, the minimal link of an almost complete intersection ideal defines a Gorenstein ring. For other favorable properties of minimal links, see \cite[Remark 2.7]{Huneke-Ulrich87}.} since the regular sequence defining the link is generated in degree $d+1$. We recall that the \textit{universal link}\footnote{The point of defining the universal link is that it is \textit{essentially a deformation} of any ideal linked to $I$ and thus controls the algebro-geometric properties of the linked ideals; see \cite{HunekeUlrich-Duke}.} of $I$ is the localization of the generic link $J\subset S$ at the prime ideal $(\underline{x_{i,j}})$ of $S$. Clearly, the universal link of a homogeneous ideal of a standard graded polynomial ring generated in the same degree is a minimal link. Since $J\subset (\underline{x_{i,j}})$, we get  

\begin{cor}\label{cor:UniversalLink}
The symbolic and ordinary powers of the universal link $J_{(\underline{x_{i,j}})}$ agree, i.e. $(J_{(\underline{x_{i,j}})})^{(\ell)}=(J_{(\underline{x_{i,j}})})^\ell$ for all $\ell \in \mathbb{N}$.
\end{cor}
While Corollary \ref{cor:UniversalLink} is weaker than Corollary \ref{cor:OrdinaryEqualsSymbolic}, we do not know how to prove it directly because the indeterminates $Y_{i.j}$ are units in the universal link and thus do not contribute to the term order!

We end this section by noting that, along the way, we characterized precisely when the symbolic and ordinary powers of the ideal of the nontrivial generators of $\ini(J)$ are equal.

\begin{cor}\label{cor:NontrivialGenerators}
The following are equivalent:
\begin{enumerate}[\quad\rm(1)]
\item $N^{(\ell)}=N^{\ell}$ for all $\ell\in\NN$.
\item $N^{(2)}=N^2$.
\item $m\leq 2$ or $m= n-1$.
\end{enumerate}
\end{cor}
\begin{proof}
$(1)\implies (2)$ is obvious and $(2)\implies (3)$ was observed in Example \ref{exU}. For $(3)\implies (1)$, by Theorem \ref{t:useful}, using Corollary \ref{c:Upowers}, it is enough to show that 
    $$
    \gamma:=\prod_{(i,j)\in [m]\times [n]}x_{i,j} \cdot \prod_{k=1}^{2r+1}\beta_{A_k}\in \big(N^{r+1}\big)^{[2]}
    $$
whenever $r\in\NN$ and $A_1\leq A_2\leq \ldots \leq A_{2r+1}$. Set \[\delta:=\prod_{k=1}^{r+1}\beta_{A_{2k-1}}\in N^{r+1},\] then $\delta^2|\gamma$ as in situation (i) of the case ``$a=0$" of the proof of Theorem \ref{thm:main}. This is due to the fact that if $m\leq 2$ or $m= n-1$, then situation (ii) in the case ``$a=0$" cannot occur.
\end{proof}

%%%%%%%%%%%%%%%%%%%%%%%%%%%%%%%%%%%%%%%%%%%%%%%%%%%%%%%%%%%%%%%%%%%
\section{Singularities of the blowup algebras of the generic link}\label{Section5}
%%%%%%%%%%%%%%%%%%%%%%%%%%%%%%%%%%%%%%%%%%%%%%%%%%%%%%%%%%%%%%%%%%%

Recall that the Rees algebra and the associated graded ring of an ideal $U$ of a commutative ring are defined as
\[\mathcal{R}(U)\colonequals \bigoplus_{\ell \geq 0}U^\ell \quad \text{and} \quad \gr(U) \colonequals \bigoplus_{\ell \geq 0}U^\ell/U^{\ell+1}.\]

We remind the reader that we stick to the notation in Setup \ref{setup}. In particular, unless specified otherwise, $I$ denotes the ideal of maximal minors of a generic matrix $X$ in the polynomial ring $R=K[X]$ over a field $K$ and $J$ is its generic link in a polynomial extension $S$ of $R$. Proposition \ref{prop:crucialmethod} and Theorem \ref{thm:main} have some immediate consequences on the blowup algebras of the generic link: 

\begin{cor}\label{cor:Rees}
The Rees algebra $\mathcal{R}(J)$ is $F$-rational in positive characteristic and has rational singularities in characteristic 0. In particular, both $\mathcal{R}(J)$ and $\gr(J)$  are Cohen--Macaulay.
\end{cor}
\begin{proof}
It only remains to justify that $\gr(J)$ is Cohen--Macaulay. This is simply because both the polynomial ring $S$ and the Rees algebra $\mathcal{R}(J)$ are Cohen--Macaulay \cite[Proposition 1.1]{HunekeIllinois}.
\end{proof}

\begin{cor}
Over an $F$-finite field of positive characteristic, both $\mathcal{R}(J)$ and $\gr(J)$ are $F$-split.
\end{cor}

\begin{rmk}
Let $I$ be an unmixed ideal in a polynomial ring $R$. If the pair $(R,I)$ has the property $\mathbf{P}$, then over an $F$-finite field, $I$ and its generic link $J$ are both symbolic $F$-split ideals by \cite[Definition 5.2, Lemma 6.2]{dS-Mn-NB} in view of Lemma \ref{lemma:Pinherited}. In particular, the symbolic Rees algebra $\mathcal{R}^s(J)$ and the symbolic associated graded algebra $\gr ^s(J)$ are $F$-split rings, though they may not be Noetherian. 

However, we note that the notion of symbolic $F$-split (or even $F$-split) does \textit{not} pass to links. Indeed, let $I$ be the ideal generated by the $2$-minors of a generic $4\times 4$ matrix $X$ in $R\colonequals K[X]$. Then, by the arguments in \cite[Lemma 4.8, Example 4.9]{Pandey}, its generic link $J$ (in a polynomial extension $R[Y]$) is not even $F$-injective. It follows that the associated graded ring $\gr(J)$ is not $F$-split. This is because $R[Y]/J$ is a summand of $\gr(J)$ and the property of being $F$-split is inherited by summands. However, in what follows, we show in particular that $\mathcal{R}(J)$ (and therefore $\gr(J)$) is Cohen--Macaulay and $J^{(\ell)}=J^\ell$ for all $\ell$.     
\end{rmk}

Next, we want to strengthen the above result and show that in positive characteristic both the Rees algebra and the associated graded ring of $J$ are $F$-regular.

\begin{rmk}
Observe that the Gr\"obner deformation argument of Proposition \ref{prop:crucialmethod}---used to prove Corollary \ref{cor:Rees}---does not allow us to conclude that the Rees algebra $\mathcal{R}(J)$ is $F$-regular since $F$-regularity does \textit{not} deform due to \cite{Singh}. The technique of $F$-split filtrations developed recently by De Stefani, Monta\~no, and  N\'u\~nez-Betancourt provides a sharper tool to study the Frobenius singularities of blowup algebras (see \cite[\S 4]{dS-Mn-NB} for details). Indeed, using this technique, and our main result Corollary \ref{cor:OrdinaryEqualsSymbolic}, we obtain the following strengthening of Corollary \ref{cor:Rees}. This result may be seen as an analogue of \cite[Theorem 6.7]{dS-Mn-NB}.
\end{rmk}

\begin{thm}\label{thn:symbolicFsplit}
The Rees algebra $\calR(J)$ and the associated graded ring $\gr(J)$ of the generic link $J$ of the ideal of maximal minors are strongly $F$-regular if $K$ is an $F$-finite field of positive characteristic, and they have rational singularities if $K$ is a field of characteristic $0$.
\end{thm}

\begin{proof}
The characteristic $0$ statement follows from the positive characteristic one by \cite{Smithrat}, so we assume that $K$ is an $F$-finite field of characteristic $p>0$.

We proceed by induction on $m$. If $m=1$, the assertion essentially follows from \cite{PT24}: If $n=1$, the generic link $J$ is generated by one variable and the assertion follows trivially. If $n>1$, $J$ by \cite[Lemmas 5.3, 5.5]{PT24} and \cite[Corollary 5.10]{dS-Mn-NB} we can choose an $F$-splitting of $\calR(J)$ and $\gr(J)$ sending $x_1:=x_{1,1}$ to $1$. Then $\mathcal{R}(J)_{x_1}=\calR(JS_{x_1})$ and $\gr(J)_{x_1}=\gr(JS_{x_1})$ are strongly $F$-regular since
\[JS_{x_1} = \mathfrak{a}S_{x_1}:(x_1,\dots,x_n)S_{x_1} = \mathfrak{a}S_{x_1}: S_{x_1} = \mathfrak{a}S_{x_1}\]
is generated by linear forms of $S_{x_1}=K[x_1,x_1^{-1}][x_2,\ldots ,x_n,Y_1,\ldots ,Y_n]$; indeed we get that $\gr(JS_{x_1})$ is a polynomial ring in $n$ variables over the regular ring $S_{x_1}/JS_{x_1}$, while $\calR(JS_{x_1})$ is a localization of a polynomial extension of a determinantal ring of $2$-minors of a generic $2\times n$-matrix. Now the $F$-regularity of $\mathcal{R}(J)$ and $\gr(J)$ follows from \cite[Theorem 5.9(a)]{Hochster-Huneke94a}. 

We now assume that $m>1$. Notice that the product $f$ of the generators of $\mathfrak{a}$ has a squarefree initial term, more precisely $\ini_{<_1}(f)=\prod_{i=1}^{n-m+1}Y_{i,i}\alpha_i$ (using the notation of Setup \ref{setup}), and $x_{1,1}$ does not divide $\ini_{<_1}(f)$. In view of this, let $\xi$ be the product of all the variables of $S$ not dividing $x_{1,1}\cdot \ini_{<_1}(f)$ and set
\[ g \colonequals \xi\cdot f.\]
Note that $g^{p-1} \in J^{[p]}:J$. It is easily checked that the Frobenius trace map
\[\varphi = \Tr(F_*(x_{1,1}^{p-2}\cdot g^{p-1})-): F_*(S) \to S. \]
gives a Frobenius splitting of $S$ such that $\varphi(F_*(x_{1,1}))=1$ and $\varphi(J)\subset J$.
As the polynomial $g \in J^{n-m+1}=J^{(n-m+1)}$, by \cite[Corollary 5.10]{dS-Mn-NB}, we have
\[g^{p-1} \in (J^{r+1})^{[p]}:J^{rp+1} \quad \forall \;r\in\NN.\]
Consequently, $\varphi$ induces maps $\Psi: F_*(\mathcal{R}(J))\to \mathcal{R}(J)$ and $\overline{\Psi}: F_*(\gr(J))\to \gr(J)$ which give $F$-splittings such that
\[\Psi(F_*(x_{1,1}))=1 \quad \text{and} \quad \overline{\Psi}(F_*(\overline {x_{1,1}}))=1. \]

Using the same arguments in the proof of \cite[Theorem 6.3]{PT24}, it turns out that the localizations $\mathcal{R}(J)_{x_{1,1}}$ and $\gr(J)_{x_{1,1}}$ are respectively faithfully flat extensions of the Rees algebra and the associated graded ring of the generic link of the ideal generated by the $(m-1)$-minors of a generic $(m-1)\times(n-1)$ matrix. Since $F$-regularity is unaffected by faithfully flat maps, we get that the rings $\mathcal{R}(J)_{x_{1,1}}$ and $\gr(J)_{x_{1,1}}$ are strongly $F$-regular by the induction hypothesis. Once again, the $F$-regularity of $\mathcal{R}(J)$ and $\gr(J)$ follows from \cite[Theorem 5.9(a)]{Hochster-Huneke94a}. 
\end{proof}

\begin{cor}\label{cor:Gor}
In any characteristic, $\calR(J)$ is a Cohen--Macaulay normal domain and $\gr (J)$ is a Gorenstein normal domain.    
\end{cor}

\begin{proof}
That $\calR(J)$ and $\gr(J)$ are both normal Cohen--Macaulay domains is a consequence of Theorem \ref{thn:symbolicFsplit}.

Since $\gr(J)$ is isomorphic to the quotient of the extended Rees algebra $\calR '(J)=S[JT,T^{-1}]$ by the ideal generated by the regular element $T^{-1}$, $\calR '(J)$ is Cohen--Macaulay and $T^{-1}$ generates a prime ideal of $\calR '(J)$. On the other hand,
 \[\calR '(J)_{T^{-1}} \cong S[T,T^{-1}] \]
 is a unique factorization domain (UFD); but then by \cite{Nagata}, $\calR '(J)$ is a UFD. Since a homomorphic image of a regular ring which is Cohen--Macaulay and a UFD is Gorenstein by \cite{Murthy}, we get that $\calR '(J)$ is a Gorenstein domain. So
 \[\calR '(J)/T^{-1}\calR '(J) \cong \gr(J)\]
 is also a Gorenstein domain, and this concludes the proof.
\end{proof}

We do not know if the symbolic and ordinary powers of the generic link of the non-maximal minors of a generic matrix are equal; the  computations with a machine become inaccessible very quickly. However, the case of a square matrix---where the generic link is an almost complete intersection---is subsumed in the following more general result pointed out to us by Ulrich:
\begin{prop}\label{prop:general}
Let $I$ be an unmixed homogeneous ideal of height $g>0$ in a polynomial ring $R$ over a field satisfying the following conditions:
\begin{itemize}
    \item $R/I$ is a Gorenstein ring.
    \item $I$ is a complete intersection in codimension $g+1$, i.e.,
    $I_{\frakp}$ is generated by an $R_{\frakp}$-regular sequence for all primes $\frakp$ such that $\dim((R/I)_{\frakp})\leq 1$.
\end{itemize} 
Let $J$ be the generic link of $I$ (in a polynomial extension $S$ of $R$). Then we have
\begin{enumerate}[\quad\rm(1)]
    \item $\mathcal{R}(J)$ is a Cohen--Macaulay normal domain.
    \item $J^{(\ell)}=J^\ell$ for all $\ell \in \NN$.
    \item $\gr(J)$ is a Gorenstein domain.
\end{enumerate}
\end{prop}

\begin{proof}
We have $IS \sim _{\mathfrak{a}} J$ (where $\fraka$ is the generic regular sequence in $IS$). By \cite{PS74}, the canonical module of $S/IS$ is
\[\omega_{S/IS} \cong J/\fraka. \]
Since $S$ is a faithfully flat extension of $R$, so  $S/IS$ is a Gorenstein ring. It follows that the linked ideal $J$ is an \textit{almost complete intersection}, i.e., $\mu(J)= \height(J)+1$. Since $I$ is a complete intersection in codimension $g+1$, by \cite[Proposition 2.9 b)]{Huneke-Ulrich85}, the generic link $J$ also has this property. Therefore
\[\mu( JS_{\mathfrak{p}})\leq \height \mathfrak{p}\]
for all prime ideals $\mathfrak{p}$ containing $J$.
In addition, since linkage preserves the Cohen--Macaulay property, $S/J$ is Cohen--Macaulay; so by \cite[Lemma 1.13]{HunekeSCM} each Koszul homology module $H_i(J;S)$ of $J$ is Cohen--Macaulay. Therefore the Rees algebra $\mathcal{R}(J)$ is also Cohen--Macaulay by \cite[Proposition 1.15]{HunekeSCM}.

To show the equality $J^{(\ell)}=J^\ell$, we need to prove that the ring $S/J^\ell$ has no embedded primes, i.e., $(S/J^\ell)_{\mathfrak{p}}$ has positive depth for all prime ideals $\mathfrak{p}$ containing $J$ such that $\height(\frakp)>g$ and for any $\ell\in \NN$. Fix such a prime $\mathfrak{p}$. Since $J$ is an almost complete intersection and a complete intersection in codimension $g+1$, we have
\[\mu( JS_{\mathfrak{p}})< \height \mathfrak{p}.\]
In particular, the analytic spread of $JS_{\mathfrak{p}}$ is smaller than $\dim S_{\mathfrak{p}}$ (see \cite{NR54} for definition and properties of the analytic spread). On the other hand 
\[\min\{\depth((S/J^\ell)_{\mathfrak{p}}:\ell \in\NN\}=\mathrm{grade} (\mathfrak{p}\gr (JS_{\mathfrak{p}}))\]
(e.g. see \cite[Proposition 9.23]{BrunsVetter}). Since $\gr(JS_{\mathfrak{p}})=\gr(J)_{\mathfrak{p}}$ is Cohen--Macaulay by the previous part, the analytic spread of $JS_{\mathfrak{p}}$, that by definition is the Krull dimension of $\gr(JS_{\mathfrak{p}})/\mathfrak{p}\gr(JS_{\mathfrak{p}})$, is 
\[\dim \gr(JS_{\mathfrak{p}})-\height(\mathfrak{p}\gr(JS_{\mathfrak{p}}))=\dim S_{\mathfrak{p}}-\mathrm{grade}(\mathfrak{p}\gr(JS_{\mathfrak{p}})).\]
Putting things together, we get that \[\depth((S/J^\ell)_{\mathfrak{p}})>0 \quad \text{for all}\quad \ell\in\NN.\] Therefore $J^{(\ell)}=J^\ell$ for all $\ell \in \NN$. Now the normality of $\mathcal{R}(J)$ follows from the fact that the generic link $J$ is a prime ideal (since $I$ is unmixed) by \cite[Proposition 2.6]{Huneke-Ulrich85} and the symbolic powers of a prime ideal in a regular ring are integrally closed.

We now show that $\gr(J)$ is a Gorenstein domain. First  notice that $\gr(J)$ is Cohen--Macaulay since $\calR(J)$ s Cohen--Macaulay \cite[Proposition 1.1]{HunekeIllinois}. Since $J^{(\ell)}=J^\ell$ for all $\ell\in\NN$, the element $T^{-1}$ of the extended Rees algebra $\calR '(J)$ generates a prime ideal by  \cite{Hochster-symbolic}, so $\gr(J)\cong \calR '(J)/T^{-1}\calR '(J)$ is a domain. Now the rest of the proof is exactly the same as that of Corollary \ref{cor:Gor}.
\end{proof}

The symbolic and ordinary powers of the non-maximal minors---of size atleast $2$---of a generic matrix are \textit{never} equal \cite[Theorem 10.4]{BrunsVetter}. However, from the above discussion, we get: 

\begin{cor}\label{cor:non-maximal}
Let $X$ be an $n \times n$ matrix of indeterminates, $K$ a field, and $R= K[X]$. Let $I$ be the ideal of $R$ generated by the $t$-minors of $X$ with $1\leq t \leq n$ and $J$ be the generic link of $I$. Then
\begin{enumerate}[\quad\rm(1)]
    \item $\mathcal{R}(J)$ is a Cohen--Macaulay normal domain.
    \item $J^{(\ell)}=J^\ell$ for all $\ell \in \NN$.
    \item $\gr(J)$ is a Gorenstein domain.
\end{enumerate}    
\end{cor}
\begin{proof}
It is well-known that $R/I$ is a Gorenstein normal domain (e.g. see \cite[Theorem 2.11, Corollary 2.21]{BrunsVetter}). Normality forces $I$ is to be a complete intersection in codimension $\height(I)+1$. Hence this is a particular case of Proposition \ref{prop:general}.
\end{proof}

%%%%%%%%%%%%%%%%%%%%%%%%%%%%%
\section*{Acknowledgments}
%%%%%%%%%%%%%%%%%%%%%%%%%%%%%

We thank Bernd Ulrich for several insightful discussions. We are especially grateful to him for pointing out Proposition \ref{prop:general}. We thank the referee for their insightful remarks which helped improve our paper substantially. 

A part of this work was done while the second author was in residence at the Simons Laufer Mathematical Sciences Institute (formerly MSRI) in Berkeley, California during the Spring 2024 semester. He thanks SLMath for its warm hospitality and excellent working conditions.

\bibliographystyle{alpha}
\bibliography{main}

@article {Hochster-symbolic,
    AUTHOR = {Hochster, Melvin},
     TITLE = {Criteria for equality of ordinary and symbolic powers of
              primes},
   JOURNAL = {Math. Z.},
  FJOURNAL = {Mathematische Zeitschrift},
    VOLUME = {133},
      YEAR = {1973},
     PAGES = {53--65},
      ISSN = {0025-5874,1432-1823},
   MRCLASS = {13A15},
  MRNUMBER = {323771},
MRREVIEWER = {L.\ J.\ Ratliff, Jr.},
       DOI = {10.1007/BF01226242},
       URL = {https://doi.org/10.1007/BF01226242},
}

@article {Murthy,
    AUTHOR = {Murthy, M. Pavaman},
     TITLE = {A note on factorial rings},
   JOURNAL = {Arch. Math. (Basel)},
  FJOURNAL = {Archiv der Mathematik},
    VOLUME = {15},
      YEAR = {1964},
     PAGES = {418--420},
      ISSN = {0003-889X,1420-8938},
   MRCLASS = {13.93 (18.00)},
  MRNUMBER = {173695},
MRREVIEWER = {Alex\ Rosenberg},
       DOI = {10.1007/BF01589225},
       URL = {https://doi.org/10.1007/BF01589225},
}

@article {Nagata,
    AUTHOR = {Nagata, Masayoshi},
     TITLE = {A remark on the unique factorization theorem},
   JOURNAL = {J. Math. Soc. Japan},
  FJOURNAL = {Journal of the Mathematical Society of Japan},
    VOLUME = {9},
      YEAR = {1957},
     PAGES = {143--145},
      ISSN = {0025-5645,1881-1167},
   MRCLASS = {09.3X},
  MRNUMBER = {84490},
MRREVIEWER = {Melvin\ Henriksen},
       DOI = {10.2969/jmsj/00910143},
       URL = {https://doi.org/10.2969/jmsj/00910143},
}

@article {Smithrat,
    AUTHOR = {Smith, Karen E.},
     TITLE = {{$F$}-rational rings have rational singularities},
   JOURNAL = {Amer. J. Math.},
  FJOURNAL = {American Journal of Mathematics},
    VOLUME = {119},
      YEAR = {1997},
    NUMBER = {1},
     PAGES = {159--180},
      ISSN = {0002-9327,1080-6377},
   MRCLASS = {13A35 (13D45 13F40 14B05)},
  MRNUMBER = {1428062},
MRREVIEWER = {Ian\ M.\ Aberbach},
       URL =
              {http://muse.jhu.edu/journals/american_journal_of_mathematics/v119/119.1smith.pdf},
}

@article {Na86,
    AUTHOR = {Narasimhan, Himanee},
     TITLE = {The irreducibility of ladder determinantal varieties},
   JOURNAL = {J. Algebra},
  FJOURNAL = {Journal of Algebra},
    VOLUME = {102},
      YEAR = {1986},
    NUMBER = {1},
     PAGES = {162--185},
      ISSN = {0021-8693},
   MRCLASS = {14M12},
  MRNUMBER = {853237},
MRREVIEWER = {Tadeusz\ J\'ozefiak},
       DOI = {10.1016/0021-8693(86)90134-1},
       URL = {https://doi.org/10.1016/0021-8693(86)90134-1},
}

@article {HHTZ08,
    AUTHOR = {Herzog, J\"urgen and Hibi, Takayuki and Trung, Ng\^o{} Vi\^et
              and Zheng, Xinxian},
     TITLE = {Standard graded vertex cover algebras, cycles and leaves},
   JOURNAL = {Trans. Amer. Math. Soc.},
  FJOURNAL = {Transactions of the American Mathematical Society},
    VOLUME = {360},
      YEAR = {2008},
    NUMBER = {12},
     PAGES = {6231--6249},
      ISSN = {0002-9947,1088-6850},
   MRCLASS = {13F55 (05C65 13A30 55U10)},
  MRNUMBER = {2434285},
MRREVIEWER = {Paulo\ F.\ Machado},
       DOI = {10.1090/S0002-9947-08-04461-9},
       URL = {https://doi.org/10.1090/S0002-9947-08-04461-9},
}

@article {NR54,
    AUTHOR = {Northcott, Douglas and Rees, David},
     TITLE = {Reductions of ideals in local rings},
   JOURNAL = {Proc. Cambridge Philos. Soc.},
  FJOURNAL = {Proceedings of the Cambridge Philosophical Society},
    VOLUME = {50},
      YEAR = {1954},
     PAGES = {145--158},
      ISSN = {0008-1981},
   MRCLASS = {09.1X},
  MRNUMBER = {59889},
MRREVIEWER = {P.\ Samuel},
       DOI = {10.1017/s0305004100029194},
       URL = {https://doi.org/10.1017/s0305004100029194},
}

@article {CoVa,
    AUTHOR = {Conca, Aldo and Varbaro, Matteo},
     TITLE = {Square-free {G}r\"obner degenerations},
   JOURNAL = {Invent. Math.},
  FJOURNAL = {Inventiones Mathematicae},
    VOLUME = {221},
      YEAR = {2020},
    NUMBER = {3},
     PAGES = {713--730},
      ISSN = {0020-9910,1432-1297},
   MRCLASS = {13P10},
  MRNUMBER = {4132955},
MRREVIEWER = {Haohao\ Wang},
       DOI = {10.1007/s00222-020-00958-7},
       URL = {https://doi.org/10.1007/s00222-020-00958-7},
}

@article {PS74,
    AUTHOR = {Peskine, Christian and Szpiro, Lucien},
     TITLE = {Liaison des vari\'et\'es alg\'ebriques. {I}},
   JOURNAL = {Invent. Math.},
  FJOURNAL = {Inventiones Mathematicae},
    VOLUME = {26},
      YEAR = {1974},
     PAGES = {271--302},
      ISSN = {0020-9910,1432-1297},
   MRCLASS = {14M10},
  MRNUMBER = {364271},
MRREVIEWER = {G.\ Horrocks},
       DOI = {10.1007/BF01425554},
       URL = {https://doi.org/10.1007/BF01425554},
}

@article {Jonathon-Luis,
    AUTHOR = {Monta{\~n}o, Jonathan and N{\'{u}}{\~n}ez-Betancourt, Luis},
     TITLE = {Splittings and symbolic powers of square-free monomial ideals},
   JOURNAL = {Int. Math. Res. Not. IMRN},
  FJOURNAL = {International Mathematics Research Notices. IMRN},
      YEAR = {2021},
    NUMBER = {3},
     PAGES = {2304--2320},
      ISSN = {1073-7928,1687-0247},
   MRCLASS = {13A30 (13A35)},
  MRNUMBER = {4206614},
MRREVIEWER = {Shreedevi\ Kalyani\ Masuti},
       DOI = {10.1093/imrn/rnz138},
       URL = {https://doi.org/10.1093/imrn/rnz138},
}

@article {Sullivant,
    AUTHOR = {Sullivant, Seth},
     TITLE = {Combinatorial symbolic powers},
   JOURNAL = {J. Algebra},
  FJOURNAL = {Journal of Algebra},
    VOLUME = {319},
      YEAR = {2008},
    NUMBER = {1},
     PAGES = {115--142},
      ISSN = {0021-8693,1090-266X},
   MRCLASS = {13A15 (05E99 13C40 13P10)},
  MRNUMBER = {2378064},
MRREVIEWER = {Benjamin\ P.\ Richert},
       DOI = {10.1016/j.jalgebra.2007.09.024},
       URL = {https://doi.org/10.1016/j.jalgebra.2007.09.024},
}

@book {BCRV,
    AUTHOR = {Bruns, Winfried and Conca, Aldo and Raicu, Claudiu and
              Varbaro, Matteo},
     TITLE = {Determinants, {G}r\"obner bases and cohomology},
    SERIES = {Springer Monographs in Mathematics},
 PUBLISHER = {Springer, Cham},
      YEAR = {[2022] \copyright 2022},
     PAGES = {xiii+507},
      ISBN = {978-3-031-05479-2; 978-3-031-05480-8},
   MRCLASS = {13P10 (13C40)},
  MRNUMBER = {4627943},
       DOI = {10.1007/978-3-031-05480-8},
       URL = {https://doi.org/10.1007/978-3-031-05480-8},
}

@article {Varbaro-Koley,
    AUTHOR = {Koley, Mitra and Varbaro, Matteo},
     TITLE = {Gr\"{o}bner deformations and {$F$}-singularities},
   JOURNAL = {Math. Nachr.},
  FJOURNAL = {Mathematische Nachrichten},
    VOLUME = {296},
      YEAR = {2023},
    NUMBER = {7},
     PAGES = {2903--2917},
      ISSN = {0025-584X,1522-2616},
   MRCLASS = {13A35 (13P10)},
  MRNUMBER = {4626865},
}

@article {Hochster-Huneke94a,
    AUTHOR = {Hochster, Melvin and Huneke, Craig},
     TITLE = {{$F$}-regularity, test elements, and smooth base change},
   JOURNAL = {Trans. Amer. Math. Soc.},
  FJOURNAL = {Transactions of the American Mathematical Society},
    VOLUME = {346},
      YEAR = {1994},
    NUMBER = {1},
     PAGES = {1--62},
      ISSN = {0002-9947,1088-6850},
   MRCLASS = {13A35 (13B99 13F40)},
  MRNUMBER = {1273534},
MRREVIEWER = {Ian\ M.\ Aberbach},
       DOI = {10.2307/2154942},
       URL = {https://doi.org/10.2307/2154942},
}

@article{Hoc,
  doi = {10.2307/1970791},
  url = {https://doi.org/10.2307/1970791},
  year = {1972},
  month = sep,
  publisher = {{JSTOR}},
  volume = {96},
  number = {2},
  pages = {318},
  author = {Melvin Hochster},
  title = {Rings of {I}nvariants of {T}ori,  {C}ohen-{M}acaulay {R}ings {G}enerated by {M}onomials,  and {P}olytopes},
  journal = {The Annals of Mathematics}
}

@article{PT24,
    AUTHOR = {Pandey, Vaibhav and Tarasova, Yevgeniya},
     TITLE = {Linkage and {$F$}-regularity of determinantal rings},
   JOURNAL = {Int. Math. Res. Not. IMRN},
  FJOURNAL = {International Mathematics Research Notices. IMRN},
      YEAR = {2024},
    NUMBER = {11},
     PAGES = {9323--9339},
      ISSN = {1073-7928,1687-0247},
   MRCLASS = {13A35 (13C40)},
  MRNUMBER = {4756115},
       DOI = {10.1093/imrn/rnae040},
       URL = {https://doi.org/10.1093/imrn/rnae040},
}

@article{Seccia,
    AUTHOR = {Seccia, Lisa},
     TITLE = {Knutson ideals and determinantal ideals of {H}ankel matrices},
   JOURNAL = {J. Pure Appl. Algebra},
  FJOURNAL = {Journal of Pure and Applied Algebra},
    VOLUME = {225},
      YEAR = {2021},
    NUMBER = {12},
     PAGES = {Paper No. 106788, 17},
      ISSN = {0022-4049,1873-1376},
   MRCLASS = {13A35 (13C40 13P10)},
  MRNUMBER = {4260034},
MRREVIEWER = {Dinh\ Thanh\ Trung},
       DOI = {10.1016/j.jpaa.2021.106788},
       URL = {https://doi.org/10.1016/j.jpaa.2021.106788},
}

@article{dS-Mn-NB,
    AUTHOR = {De Stefani, Alessandro and Monta{\~n}o, Jonathan and
            N{\'{u}}{\~n}ez-Betancourt, Luis},
     TITLE = {Blowup algebras of determinantal ideals in prime
              characteristic},
   JOURNAL = {J. Lond. Math. Soc. (2)},
  FJOURNAL = {Journal of the London Mathematical Society. Second Series},
    VOLUME = {110},
      YEAR = {2024},
    NUMBER = {2},
     PAGES = {Paper No. e12969, 50},
      ISSN = {0024-6107,1469-7750},
   MRCLASS = {13A35 (13A02 13A30 13C15)},
  MRNUMBER = {4777230},
       DOI = {10.1112/jlms.12969},
       URL = {https://doi.org/10.1112/jlms.12969},
}

@article{Pandey,
AUTHOR = {Pandey,Vaibhav},
TITLE = {{$F$}-purity and the {$F$}-pure threshold as invariants of linkage},
JOURNAL ={\url{https://arxiv.org/abs/2406.05323}},
YEAR = {2024},
}

@article{Macaulay2,
    author = {Grayson, Daniel and Stillman, Michael },
    title = {Macaulay 2, a software system for research in algebraic
  geometry},
    journal = {\url{http://www.math.uiuc.edu/Macaulay2/}},
    year = {2002},
}

@article {HunekeUlrich-Duke,
    AUTHOR = {Huneke, Craig and Ulrich, Bernd},
     TITLE = {Algebraic linkage},
   JOURNAL = {Duke Math. J.},
  FJOURNAL = {Duke Mathematical Journal},
    VOLUME = {56},
      YEAR = {1988},
    NUMBER = {3},
     PAGES = {415--429},
      ISSN = {0012-7094,1547-7398},
   MRCLASS = {13H10 (13D10 14M05)},
  MRNUMBER = {948528},
MRREVIEWER = {Matthew\ Miller},
       DOI = {10.1215/S0012-7094-88-05618-9},
       URL = {https://doi.org/10.1215/S0012-7094-88-05618-9},
}

@article {Singh,
    AUTHOR = {Singh, Anurag K.},
     TITLE = {{$F$}-regularity does not deform},
   JOURNAL = {Amer. J. Math.},
  FJOURNAL = {American Journal of Mathematics},
    VOLUME = {121},
      YEAR = {1999},
    NUMBER = {4},
     PAGES = {919--929},
      ISSN = {0002-9327,1080-6377},
   MRCLASS = {13A35 (13C40 13H10)},
  MRNUMBER = {1704481},
MRREVIEWER = {Ian\ M.\ Aberbach},
       URL =
              {http://muse.jhu.edu/journals/american_journal_of_mathematics/v121/121.4singh.pdf},
}

@article {HunekeIllinois,
    AUTHOR = {Huneke, Craig},
     TITLE = {On the associated graded ring of an ideal},
   JOURNAL = {Illinois J. Math.},
  FJOURNAL = {Illinois Journal of Mathematics},
    VOLUME = {26},
      YEAR = {1982},
    NUMBER = {1},
     PAGES = {121--137},
      ISSN = {0019-2082},
   MRCLASS = {13H10},
  MRNUMBER = {638557},
MRREVIEWER = {Ming-chang\ Kang},
}

@article {Huneke-Ulrich87,
    AUTHOR = {Huneke, Craig and Ulrich, Bernd},
     TITLE = {The structure of linkage},
   JOURNAL = {Ann. of Math. (2)},
  FJOURNAL = {Annals of Mathematics. Second Series},
    VOLUME = {126},
      YEAR = {1987},
    NUMBER = {2},
     PAGES = {277--334},
      ISSN = {0003-486X,1939-8980},
   MRCLASS = {13H10 (13D10 13H15 14B07)},
  MRNUMBER = {908149},
MRREVIEWER = {Matthew\ Miller},
       DOI = {10.2307/1971402},
       URL = {https://doi.org/10.2307/1971402},
}

@article {Huneke-Ulrich85,
    AUTHOR = {Huneke, Craig and Ulrich, Bernd},
     TITLE = {Divisor class groups and deformations},
   JOURNAL = {Amer. J. Math.},
  FJOURNAL = {American Journal of Mathematics},
    VOLUME = {107},
      YEAR = {1985},
    NUMBER = {6},
     PAGES = {1265--1303},
      ISSN = {0002-9327,1080-6377},
   MRCLASS = {13D10 (13H10)},
  MRNUMBER = {815763},
MRREVIEWER = {Matthew\ Miller},
       DOI = {10.2307/2374407},
       URL = {https://doi.org/10.2307/2374407},
}

@article {HunekeSCM,
    AUTHOR = {Huneke, Craig},
     TITLE = {Linkage and the {K}oszul homology of ideals},
   JOURNAL = {Amer. J. Math.},
  FJOURNAL = {American Journal of Mathematics},
    VOLUME = {104},
      YEAR = {1982},
    NUMBER = {5},
     PAGES = {1043--1062},
      ISSN = {0002-9327,1080-6377},
   MRCLASS = {13H10 (13D25)},
  MRNUMBER = {675309},
MRREVIEWER = {Hamid\ Rahbar-Rochandel},
       DOI = {10.2307/2374083},
       URL = {https://doi.org/10.2307/2374083},
}

@book {BrunsVetter,
    AUTHOR = {Bruns, Winfried and Vetter, Udo},
     TITLE = {Determinantal rings},
    SERIES = {Lecture Notes in Mathematics},
    VOLUME = {1327},
 PUBLISHER = {Springer-Verlag, Berlin},
      YEAR = {1988},
     PAGES = {viii+236},
      ISBN = {3-540-19468-1},
   MRCLASS = {13-02 (13C05 13H10 14M12)},
  MRNUMBER = {953963},
MRREVIEWER = {Piotr\ Pragacz},
       DOI = {10.1007/BFb0080378},
       URL = {https://doi.org/10.1007/BFb0080378},
}

@incollection {Huneke-Simis-Vasconcelos,
    AUTHOR = {Huneke, Craig and Simis, Aron and Vasconcelos, Wolmer},
     TITLE = {Reduced normal cones are domains},
 BOOKTITLE = {Invariant theory ({D}enton, {TX}, 1986)},
    SERIES = {Contemp. Math.},
    VOLUME = {88},
     PAGES = {95--101},
 PUBLISHER = {Amer. Math. Soc., Providence, RI},
      YEAR = {1989},
      ISBN = {0-8218-5094-6},
   MRCLASS = {13B99 (13C10 13H05)},
  MRNUMBER = {999985},
MRREVIEWER = {L.\ J.\ Ratliff, Jr.},
       DOI = {10.1090/conm/088/999985},
       URL = {https://doi.org/10.1090/conm/088/999985},
}

@article {Huneke-Ulrich88,
    AUTHOR = {Huneke, Craig and Ulrich, Bernd},
     TITLE = {Residual intersections},
   JOURNAL = {J. Reine Angew. Math.},
  FJOURNAL = {Journal f\"{u}r die Reine und Angewandte Mathematik. [Crelle's
              Journal]},
    VOLUME = {390},
      YEAR = {1988},
     PAGES = {1--20},
      ISSN = {0075-4102,1435-5345},
   MRCLASS = {13H10 (13C13 13D10 14B07)},
  MRNUMBER = {953673},
MRREVIEWER = {J\"{u}rgen\ Herzog},
       DOI = {10.1515/crll.1988.390.1},
       URL = {https://doi.org/10.1515/crll.1988.390.1},
}

@article {DEP,
    AUTHOR = {de Concini, Corrado and Eisenbud, David and Procesi, Claudio},
     TITLE = {Young diagrams and determinantal varieties},
   JOURNAL = {Invent. Math.},
  FJOURNAL = {Inventiones Mathematicae},
    VOLUME = {56},
      YEAR = {1980},
    NUMBER = {2},
     PAGES = {129--165},
      ISSN = {0020-9910,1432-1297},
   MRCLASS = {14M12 (14L30 15A72 20C30)},
  MRNUMBER = {558865},
MRREVIEWER = {Vladimir\ L.\ Popov},
       DOI = {10.1007/BF01392548},
       URL = {https://doi.org/10.1007/BF01392548},
}

@book {Hema,
    AUTHOR = {Srinivasan, Hema},
     TITLE = {Multiplicative structures on some canonical resolutions (Resolutions, Algebra structures, Complexes)},
      NOTE = {Thesis (Ph.D.)--Brandeis University},
 PUBLISHER = {ProQuest LLC, Ann Arbor, MI},
      YEAR = {1986},
     PAGES = {124},
   MRCLASS = {99-05},
  MRNUMBER = {2635041},
       URL =
              {http://gateway.proquest.com/openurl?url_ver=Z39.88-2004&rft_val_fmt=info:ofi/fmt:kev:mtx:dissertation&res_dat=xri:pqdiss&rft_dat=xri:pqdiss:8617031},
}

@article {Chardin-Ulrich,
    AUTHOR = {Chardin, Marc and Ulrich, Bernd},
     TITLE = {Liaison and {C}astelnuovo-{M}umford regularity},
   JOURNAL = {Amer. J. Math.},
  FJOURNAL = {American Journal of Mathematics},
    VOLUME = {124},
      YEAR = {2002},
    NUMBER = {6},
     PAGES = {1103--1124},
      ISSN = {0002-9327,1080-6377},
   MRCLASS = {14M06 (13C40 13D45 14F17)},
  MRNUMBER = {1939782},
MRREVIEWER = {Philippe\ Gimenez},
       URL =
              {http://muse.jhu.edu/journals/american_journal_of_mathematics/v124/124.6chardin.pdf},
}

\end{document}